\documentclass[11pt]{amsart} \textwidth=14.5cm \oddsidemargin=1cm
\usepackage{amsmath,wasysym}
\usepackage{amsxtra}
\usepackage{amscd}
\usepackage{amsthm}
\usepackage{amsfonts}
\usepackage{amssymb}
\usepackage{eucal}
\usepackage{graphics}
\usepackage{hyperref,mathrsfs}
\usepackage[usenames,dvipsnames]{xcolor}
\usepackage{tikz}
\usepackage{tikz-cd}
\usepackage{stmaryrd}
\usepackage[all]{xy}
\usepackage{hyperref}
\usepackage{color}
\usepackage{bbm} % \mathbbm
\usepackage{upgreek}
\usepackage{extarrows}
\allowdisplaybreaks

\hypersetup{colorlinks,linkcolor=blue, urlcolor=blue,
	citecolor=blue}

\newtheorem{theorem}{Theorem} [section]
\newtheorem{lemma}[theorem]{Lemma}
\newtheorem{corollary}[theorem]{Corollary}
\newtheorem{proposition}[theorem]{Proposition}

\newtheorem{remark}[theorem]{Remark}

\numberwithin{equation}{section}

\renewcommand{\i}{\mathbf{i}}
\renewcommand{\j}{\mathbf{j}}
\newcommand{\eps}{\epsilon}
\renewcommand{\v}{\mathbf{v}}

\newcommand{\nc}{\newcommand}
\nc{\browntext}[1]{\textcolor{brown}{#1}}
\nc{\greentext}[1]{\textcolor{green}{#1}}
\nc{\redtext}[1]{\textcolor{red}{#1}}
\nc{\bluetext}[1]{\textcolor{blue}{#1}}
\nc{\brown}[1]{\browntext{ #1}}
\nc{\green}[1]{\greentext{ #1}}
\nc{\red}[1]{\redtext{ #1}}
\nc{\blue}[1]{\bluetext{ #1}}

\title[Howe duality between quantum general linear supergroups]{Three approaches to the Howe duality between quantum general linear supergroups}

\author[Li Luo]{Li Luo}
\author[Xirui Yu]{Xirui Yu}
\address{School of Mathematical Sciences, Key Laboratory of MEA (Ministry of Education) \& Shanghai Key Laboratory of PMMP, East China Normal University, Shanghai 200241, China}
\email{lluo@math.ecnu.edu.cn (Luo), 51255500105@stu.ecnu.edu.cn (Yu)}
\author[Zhongguo Zhou]{Zhongguo Zhou}
\address{School of Mathematics, Hohai University, Nanjing 210098, China}
\email{zhgzhou@hhu.edu.cn (Zhou)}

\begin{document}

\begin{abstract} 
The Howe duality between quantum general linear supergroups was firstly established by Y. Zhang via quantum coordinate superalgebras. In this paper, we provide two other approaches to this Howe duality. One is constructed by quantum differential operators, while the other is based on the Beilinson-Lusztig-MacPherson realization of $U_q(\mathfrak{gl}_{m|n})$. Moreover, we show that these three approaches are equivalent by giving their action formulas explicitly. 
\end{abstract}

\maketitle

%\tableofcontents
%===========================================
\section{Introduction}
%===========================================
\subsection{History and Goal}
The classical Howe duality, which involves commuting actions of a pair of Lie groups and Lie algebras, provides a representation-theoretic treatment for invariant theory. It has long been of interest, since its associated Howe correspondence plays a significant role in the theory of automorphic forms.

Among various generalizations, the first quantized Howe duality was established by Quesne \cite{Q92}, who provided a double centralizer property between $U_q(\mathfrak{su}_3)$ and $U_q(\mathfrak{u}_2)$. Noumi, Umeda and Wakayama \cite{NUW95,NUW96} obtained quantum analogues of the Howe dual pairs $(\mathfrak{sl}_n,\mathfrak{so}_n)$ and $(\mathfrak{sp}_2,\mathfrak{so}_n)$. 
Zhang \cite{Z03} employed quantum coordinate algebras to construct a non-commutative analogue of symmetric algebras, which accommodates a Howe duality between quantum groups $U_q(\mathfrak{gl}_m)$ and $U_q(\mathfrak{gl}_n)$. 
This formulation was also generalized in \cite{LZ03} to establish the Howe dualities of pairs $(U_q(\mathfrak{gl}_n),U_q(\mathfrak{so}_{2n}))$, $(U_q(\mathfrak{gl}_n),U_q(\mathfrak{so}_{2n+1}))$ and $(U_q(\mathfrak{gl}_n),U_q(\mathfrak{sp}_{2n}))$. In a geometric perspective, Baumann \cite{Ba99} provided a construction of $(U_q(\mathfrak{gl}_m), U_q(\mathfrak{gl}_n))$-duality through $\mathrm{GL}_d$-orbits on the product of an $m$-step and an $n$-step flag varieties; see also \cite{W01} by Wang for a non-quantum version via Lagrangian construction in a similar spirit. The equivalence between these algebraic and geometric constructions was shown in \cite{LX22}.

Baumann's geometric construction originates from Iwahori's seminal realization of the Hecke algebra (associated with an algebraic group $G$) as an involution algebra of $G$-invariant functions on the double complete flag variety (cf. \cite{IM65}). This framework underwent significant expansion in the landmark work \cite{BLM90} of Beilinson, Lusztig and MacPherson (abbr. BLM), who generalized Iwahori's construction to realize quantum Schur algebras of type $A$ by using $n$-step flags instead of complete flags. They derived a closed formula for the product of a simple generator and a basis element of Schur algebras and then established a stabilization property, by which the quantum group $U_q(\mathfrak{gl}_n)$ was realized. 

A super counterpart of the BLM realization was obtained in \cite{DG14} by Du and Gu. As an analogy, we shall generalize Baumann's construction to the super-setup in this paper. That is, to provide the Howe duality between quantum general linear supergroups along Du-Gu's way. Building on this inspiration, we may also utilize other methods to realize quantum general linear supergroups for constructing Howe duality between them. For instance, Du and Zhou \cite{DZ20} employed quantum differential operators to realize quantum general linear supergroups, likewise following which we will establish the Howe duality between quantum general linear supergroups, too. 
To be honest, such a super analogue of Howe duality has been established by Y. Zhang \cite{Zh20} via quantum coordinate superalgebras when he studied the fundamental theorems of invariant theory for quantum general linear supergroups. His formulation can be traced back to \cite{Z98}, and is an improvement of Wu and R.B. Zhang's achievement \cite{WZ09} on Howe duality for the pair $(\mathfrak{gl}_{m|n},\mathfrak{gl}_d)$. We shall also revisit Y. Zhang's work and prove the equivalence between these three different approaches. 

%----------------------------
\subsection{Mail results}
%-----------------------------
In summary, we explicitly formulate Fock spaces that carry a left $U_\v(\mathfrak{gl}_{k|l})$-action and a right $U_\v(\mathfrak{gl}_{r|s})$-action through three approaches. 

The first is via a quantum coordinate superalgebra $\mathcal{M}^{k|l}_{r|s}$ generated by some quantum coordinate coefficients (see \S\ref{sec:M}). Although Y. Zhang has given this construction in \cite{Zh20}, we have made appropriate adjustments in \eqref{eq:act1} to better accommodate the left $U_\v(\mathfrak{gl}_{k|l})$-action and the right $U_\v(\mathfrak{gl}_{r|s})$-action on $\mathcal{M}^{k|l}_{r|s}$. In addition, we provide explicit formulas for the left and right actions in detail (see Proposition~\ref{coor}), which Y. Zhang did not specify.

The second is through a mixed tensor product of polynomial superalgebras in \eqref{eq:S} to produce $\mathcal{S}^{k|l}_{r|s}$, on which the left action of $U_\v(\mathfrak{gl}_{k|l})$ and the right action of $U_\v(\mathfrak{gl}_{r|s})$ are given by quantum differential operators. These action formulas are also explicitly listed in Proposition~\ref{prop:3.3}. 

The third one arises from the convolution algebra of function spaces over orbits in certain double flag variety. Although in the super case, this geometric interpretation cannot be directly applied, Du and Gu have algebraically described the relevant geometric objects using an ``even-odd trivial intersection'' condition. Building on Du-Gu's work, we construct an $(\mathbf{S}_{\v}(k|l,d),\mathbf{S}_{\v}(r|s,d))$-bimodule $\mathcal{V}^{k|l}_{r|s}(d)$ in \S\ref{sev:V}, where detailed action formulas for the generators of Schur superalgebras $\mathbf{S}_{\v}(k|l,d)$ and $\mathbf{S}_{\v}(r|s,d)$ are provided as well. Thanks to the surjectivity of $U_\v(\mathfrak{gl}_{m|n})\to\mathbf{S}_{\v}(m|n,d)$, we further give formulas of the actions of $U_\v(\mathfrak{gl}_{k|l})$ and $U_\v(\mathfrak{gl}_{r|s})$ on $\mathcal{V}^{k|l}_{r|s}=\bigoplus_{d\geq 0}\mathcal{V}^{k|l}_{r|s}(d)$ in Proposition~\ref{prop:4.11}.

Thanks to these explicit action formulas, we see that these three $(U_\v(\mathfrak{gl}_{k|l}), U_\v(\mathfrak{gl}_{r|s}))$-bimodules $\mathcal{M}^{k|l}_{r|s}$, $\mathcal{S}^{k|l}_{r|s}$ and $\mathcal{V}^{k|l}_{r|s}$ are isomorphic. Therefore, they all admit a Howe duality between $U_\v(\mathfrak{gl}_{k|l})$ and $U_\v(\mathfrak{gl}_{r|s})$, as it was shown in \cite{Zh20} that $\mathcal{M}^{k|l}_{r|s}$ does.

%--------------------------
\subsection{Organization}
%--------------------------
Here is a layout of the paper. In Section 2, we revisit the quantum coordinate superalgebras formulation for the super Howe duality. In Section 3 we provide its differential operator construction, while in Section 4 we give its geometric BLM realization. Our explicit action formulas show that these three approaches are equivalent.

\subsubsection*{Acknowledgement} We thank the referee for a helpful suggestion that allowed us to significantly shorten the proof of Proposition~\ref{left}.
The work is partially supported by the National Key R\&D Program of China (No. 2024YFA1013802), the NSF of China (No. 12371028) and the Science and Technology Commission of Shanghai Municipality (No. 22DZ2229014).

%============================================
\section{Quantum coordinate superalgebra construction}\label{sec:2}
%===========================================
%------------------------------------
\subsection{Quantum general linear supergroup}
%------------------------------------
Given $m,n\in\mathbb{N}$, denote
$$I_{m|n}=\{1,2,\ldots,m+n\}.$$ Let $\mathbb{Z}_2=\{\overline{0},\overline{1}\}$ be the finite field with two elements.
For an integer $a\in I_{m|n}$, we define its parity by
\begin{equation*}
\widehat{a}:=\left\{
\begin{array}{ll}
\overline{0},& 1\leq a\leq m;\\
\overline{1},& m+1\leq a\leq m+n
\end{array}
\right.
\end{equation*}

Let $\v$ be an indeterminate.
For $a\in I_{m|n}$, denote
\begin{equation*}
\v_a:=\left\{\begin{array}{ll} 
 \v, & 1\leq a\leq m;\\
\v^{-1}, & m+1\leq a\leq m+n.
 \end{array}\right.
 \end{equation*}

The quantum general linear supergroup $U_{\v}(\mathfrak{gl}_{m|n})$ is a unital associative superalgebra over $\mathbb{C}(\v)$ generated by
\begin{itemize}
\item even generators $K_a, K^{-1}_a (a\in I_{m|n}), E_{b,b+1}, E_{b+1,b} (b\in I_{m|n}, b\neq m,m+n)$;
\item odd generators $E_{m,m+1}, E_{m+1,m}$,
\end{itemize}
subject to the relations:
\begin{itemize}
\item[(R1)] $K_a K^{-1}_a=K^{-1}_a K_a=1,\quad K_a K_b=K_b K_a$;
\item[(R2)] $K_a E_{b,b\pm1} K^{-1}_a=\v_a^{\delta_{a,b}-\delta_{a,b\pm1}}E_{b,b\pm1}$;
\item[(R3)] $[E_{a,a+1},E_{b+1,b}]=\delta_{ab}\frac{K_a K^{-1}_{a+1}-K^{-1}_a K_{a+1}}{\v_a-\v^{-1}_a}$;
\item[(R4)] $E_{a,a+1}E_{b,b+1}=E_{b,b+1}E_{a,a+1},\ E_{a+1,a}E_{b+1,b}=E_{b+1,b}E_{a+1,a},\quad (a-b\geq2)$;
\item[(R5)] $(E_{a,a+1})^2E_{b,b+1}-(\v+\v^{-1})E_{a,a+1}E_{b,b+1}E_{a,a+1}+E_{b,b+1}(E_{a,a+1})^2=0$,\\
$(E_{a+1,a})^2E_{b+1,b}-(\v+\v^{-1})E_{a+1,a}E_{b+1,b}E_{a+1,a}+E_{b+1,b}(E_{a+1,a})^2=0$,\\ $(|a-b|=1, a\neq m)$;
\item[(R6)] $(E_{m,m+1})^2=(E_{m+1,m})^2=[E_{m-1,m+2},E_{m,m+1}]=[E_{m+2,m-1},E_{m+1,m}]=0$,
\end{itemize}
where $E_{m-1,m+2}$ and $E_{m+2,m-1}$ are determined inductively by
\begin{equation*}
E_{ab}=\left\{
\begin{array}{ll}
E_{ac}E_{cb}-\v^{-1}_cE_{cb}E_{ac}, & a<c<b;\\
E_{ac}E_{cb}-\v_cE_{cb}E_{ac}, & b<c<a.
\end{array}
\right.
\end{equation*}
We also denote by $\widehat{\quad}$ the parity of a homogeneous element in $U_\v(\mathfrak{gl}_{m|n})$, e.g. $$\widehat{K_a}=\widehat{K_a^{-1}}=0\quad \mbox{and} \quad \widehat{E_{ab}}=\widehat{a}+\widehat{b}.$$ 

It is known that $U_\v(\mathfrak{gl}_{m|n})$ is a Hopf superalgebra, whose coproduct $\Delta$, counit $\eps$ and antipode $S$ are given as follows:
\begin{align*}
\Delta(K_{a}^{\pm1})=K_{a}^{\pm1}\otimes K_{a}^{\pm1},\quad
&\Delta(E_{a,a+1})=E_{a,a+1}\otimes K_{a}K_{a+1}^{-1}+1\otimes E_{a,a+1},\\
&\Delta(E_{a+1,a})=E_{a+1,a}\otimes1+K_{a}^{-1}K_{a+1}\otimes E_{a+1,a};
\end{align*}
\begin{align*}
\eps(K^{\pm}_a)=1,\quad 
\eps (E_{a,a+1})=\eps (E_{a+1,a})=0;
\end{align*}
\begin{align*}
S(K_{a}^{\pm1})=K_{a}^{\mp1},\quad
S(E_{a,a+1})=-E_{a,a+1}K^{-1}_{a}K_{a+1},\quad
S(E_{a+1,a})=-K_{a}K^{-1}_{a+1}E_{a+1,a}.
\end{align*}
Here ``super'' means that the antipode $S$ is a $\mathbb{Z}_2$-graded algebra anti-automorphism, i.e. for homogeneous $x$, $y\in U_\v(\mathfrak{gl}_{m|n})$, $S(xy)=(-1)^{\widehat{x}\widehat{y}}S(y)S(x)$. We shall use Sweedler's symbol $$\Delta(x)=\sum_{(x)}x_{(1)}\otimes x_{(2)}\quad\mbox{for}\quad x\in U_\v(\mathfrak{gl}_{m|n}).$$

For $k\leq m$ and $l\leq n$, let $$\widetilde{I}_{k|l}=\{m-k+1,\ldots,m,m+1,\ldots,m+l\}.$$ 
We always regard $U_\v(\mathfrak{gl}_{k|l})$ as the subsuperalgebra of $U_\v(\mathfrak{gl}_{m|n})$ generated by $K_a^{\pm1}$, $E_{b,b+1}$ and $E_{b+1,b}$, $(a,b\in\widetilde{I}_{k|l}, b\neq m+l)$.

%--------------------
\subsection{Quantum coordinate superalgebra} \label{sec:M}
%--------------------
Let $V$ be the natural module of $U_\v(\mathfrak{gl}_{m|n})$, where $V=V_{\overline{0}}\oplus V_{\overline{1}}$ is a superspace with a basis $\{v_a\mid a\in I_{m|n}\}$ such that $\hat{v}_a=\hat{a}$. The $U_\v(\mathfrak{gl}_{m|n})$-action on $V$ is given by   
$K_a v_{b}=\v_a^{\delta_{a,b}}v_b$ and $E_{a,a\pm1}v_b=\delta_{b,a\pm1}v_a$.

Let
\begin{equation*}
U_\v(\mathfrak{gl}_{m|n})^{\circ}:=\{f\in(U_\v(\mathfrak{gl}_{m|n}))^* \mid \mbox{$\mathrm{Ker} f$ contains a cofinite ideal of $U_\v(\mathfrak{gl}_{m|n})$}\}
\end{equation*}
be the finite dual of $U_\v(\mathfrak{gl}_{m|n})$, which is also a Hopf superalgebra.
We have the matrix elements $t_{ij}\in U_\v(\mathfrak{gl}_{m|n})^{\circ}, (i,j\in I_{m|n})$, defined by
\begin{equation*}
\langle t_{ij},x \rangle=x_{ij},\quad \forall x\in U_\v(\mathfrak{gl}_{m|n}),
\end{equation*}
where $\langle\cdot,\cdot\rangle$ is the dual space pairing and $x_{ij}$ is the $(i,j)$-th matrix entry of $x$ acting on $V$ associated with the basis $\{v_a \mid a\in I_{m|n}\}$. 
The parity of $t_{ij}$ is defined as $\hat{t}_{ij}=\hat{i}+\hat{j}$.

In \cite{Z98}, R.B. Zhang introduced a subsuperalgebra $\mathcal{M}_{m|n}$ of $U_\v(\mathfrak{gl}_{m|n})^\circ$, which is generated by $t_{ij}$, $(i,j\in I_{m|n})$ with multiplication as follows: for homogeneous elements $t, t'\in\mathcal{M}_{m|n}$ and $x\in U_\v(\mathfrak{gl}_{m|n})$,
\begin{equation*}
\langle tt',x \rangle=\sum_{(x)}\langle t\otimes t',x_{(1)}\otimes x_{(2)} \rangle=\sum_{(x)}(-1)^{\hat{t'}\hat{x}_{(1)}}\langle t,x_{(1)}\rangle \langle t',x_{(2)} \rangle.
\end{equation*} 

\begin{lemma}{\cite[Lemma 2.6]{Zh20}}
The superalgebra $\mathcal{M}_{m|n}$ is generated by the matrix elements $t_{ij}$ satisfies the following defining relations:
\begin{align*}
&(t_{ij})^2=0,&&\hat{i}+\hat{j}=\overline{1};\\
&t_{ia}t_{ja}=(-1)^{(\hat{i}+\hat{a})(\hat{j}+\hat{a})}\v_{a}t_{ja}t_{ia},&&i>j;\\
&t_{ai}t_{aj}=(-1)^{(\hat{a}+\hat{i})(\hat{a}+\hat{j})}\v_{a}t_{aj}t_{ai},&&i>j;\\
&t_{ia}t_{jb}=(-1)^{(\hat{i}+\hat{a})(\hat{j}+\hat{b})}t_{jb}t_{ia},&&i>j,a<b;\\
&t_{ia}t_{jb}=(-1)^{(\hat{i}+\hat{a})(\hat{j}+\hat{b})}t_{jb}t_{ia}+(-1)^{\hat{i}(\hat{j}+\hat{b})+\hat{j}\hat{b}}(\v-\v^{-1})t_{ja}t_{ib},&&i>j,a>b.
\end{align*}
\end{lemma}
Moreover, the superalgebra $\mathcal{M}_{m|n}$ inherits the comultiplication $\Delta$ and the counit $\eps$ of $U_\v(\mathfrak{gl}_{m|n})^{\circ}$. Precisely,
\begin{equation*}
\Delta(t_{ij})=\sum_{a\in I_{m|n}}(-1)^{(\hat{i}+\hat{a})(\hat{a}+\hat{j})}t_{ia}\otimes t_{aj},\qquad \eps(t_{ij})=\delta_{ij}.
\end{equation*}

We denote by $M(m|n)$ the subset of $(m+n)\times (m+n)$ matrices with entries in $\mathbb{N}$. Each element in $M(m|n)$ is of the form
\begin{align*} 
A=&\begin{pmatrix} 
A_{11}&A_{12}\\
A_{21}&A_{22}
\end{pmatrix} \begin{array}{c} m\\n\end{array}\\ &\ \quad m \qquad n
\end{align*}
where the entries in blocks $A_{12}$ and $A_{21}$ are either $0$ or $1$. 

For $A=(a_{ij})\in M(m|n)$, let 
\begin{equation}\label{def:tA}
t^{(A)}=\prod_{1\leq i,j\leq m+n}^{<}t_{ij}^{a_{ij}},
\end{equation}
where ``$<$'' means that the product is arranged in such a way that $t_{ij}$ is positioned in front of $t_{kl}$ if $j<l$ or `$j=l, i<k$'. As a $\mathbb{C}(\v)$-space, $\mathcal{M}_{m|n}$ has a basis $\{t^{(A)}\mid A\in M(m|n)\}$.
For $k,r\leq m$ and $l,s\leq n$, let $\mathcal{M}^{k|l}_{r|s}$ be the subsuperalgebra of $\mathcal{M}_{m|n}$ generated by $\{t_{ij}\mid i\in\widetilde{I}_{k|l}, j\in\widetilde{I}_{r|s}\}$. It has a $\mathbb{C}(\v)$-basis $\{t^{(A)} \mid A\in M(k|l,r|s)\}$, where $M(k|l,r|s)$ is the subset of $M(m|n)$ consisting of the matrices whose $(i,j)$-th entry is $0$ unless $i\in\widetilde{I}_{k|l}$ and $ j\in\widetilde{I}_{r|s}$.
We call $\mathcal{M}_{m|n}$ and $\mathcal{M}^{k|l}_{r|s}$ quantum coordinate superalgebras.

%------------------
\subsection{Howe duality}\label{sec:2.3}
%------------------
It was initiated in \cite{Z98} and improved in \cite{Zh20} that there are left and right actions of $U_\v(\mathfrak{gl}_{m|n})$ on $\mathcal{M}_{m|n}$ as follows:
for $x\in U_\v(\mathfrak{gl}_{m|n})$ and $f\in\mathcal{M}_{m|n}$,
\begin{equation}\label{eq:act1}
x\cdot f:=\sum_{(f)}\langle f_{(1)},\omega(x)\rangle f_{(2)},
\qquad 
f\cdot x:=\sum_{(f)}f_{(1)}\langle f_{(2)},\omega(x) \rangle,
\end{equation}
where $\omega$ is the anti-involution of $U_\v(\mathfrak{gl}_{m|n})$ such that
\begin{equation}\label{anti}
K_i\mapsto K_i,\quad E_{h,h+1}\mapsto E_{h+1,h},\quad E_{h+1,h}\mapsto E_{h,h+1}.
\end{equation}
By restriction, these actions also make $\mathcal{M}^{k|l}_{r|s}$ a left $U_\v(\mathfrak{gl}_{k|l})$-module and a right $U_\v(\mathfrak{gl}_{r|s})$-module.

For example, for $h,i,j\in I_{m|n}$, we have
\begin{equation}\label{kl}
\begin{aligned}
K_{h}^{\pm1}\cdot t_{ij}=&\sum_{k\in I_{m|n}}(-1)^{(\hat{i}+\hat{k})(\hat{k}+\hat{j})}\langle t_{ik},K_{h}^{\pm1}\rangle t_{kj}=\langle t_{ii},K_{h}^{\pm1}\rangle t_{ij}=\v_{h}^{\pm\delta_{ih}}t_{ij},
\end{aligned}
\end{equation}
\begin{equation}\label{el}
\begin{aligned}
E_{m,m+1}\cdot t_{ij}=&\sum_{k\in I_{m|n}}(-1)^{(\hat{i}+\hat{k})(\hat{k}+\hat{j})}\langle t_{ik},E_{m+1,m}\rangle t_{kj}\\
=&(-1)^{(\hat{i}+\widehat{i-1})(\widehat{i-1}+\hat{j})}\langle t_{i,i-1},E_{m+1,m}\rangle t_{i-1,j}=\delta_{i,m+1}(-1)^{\hat{j}}t_{i-1,j},
\end{aligned}
\end{equation}
\begin{equation}\label{kr}
\begin{aligned}
t_{ij}\cdot K_{h}^{\pm1}=&\sum_{k\in I_{m|n}}(-1)^{(\hat{i}+\hat{k})(\hat{k}+\hat{j})}t_{ik}\langle t_{kj},K_{h}^{\pm1}\rangle=t_{ij}\langle t_{jj},K_{h}^{\pm1}\rangle=\v_{h}^{\pm\delta_{jh}}t_{ij},
\end{aligned}
\end{equation}
\begin{equation}\label{er}
\begin{aligned}
t_{ij}\cdot E_{m+1,m}=&\sum_{k\in I_{m|n}}(-1)^{(\hat{i}+\hat{k})(\hat{k}+\hat{j})}t_{ik}\langle t_{kj},E_{m,m+1}\rangle \\ 
=&(-1)^{(\hat{i}+\widehat{j-1})(\widehat{j-1}+\hat{j})}t_{i,j-1}\langle t_{j-1,j},E_{m,m+1}\rangle=\delta_{j,m+1}(-1)^{\hat{i}}t_{i,j-1}.
\end{aligned}
\end{equation}

\begin{remark}
In fact, the `$\mathcal{R}$-action' in \cite{Zh20} is a left action. We adjust it to a right action as above.  
\end{remark}

\begin{theorem}[Howe duality \cite{Zh20}]\label{Howe} 
The left $U_\v(\mathfrak{gl}_{k|l})$-action and the right $U_\v(\mathfrak{gl}_{r|s})$-action on $\mathcal{M}^{k|l}_{r|s}$ admit a double centralizer property. Hence, $\mathcal{M}^{k|l}_{r|s}$ has a multiplicity-free decomposition:
\begin{equation*}
\mathcal{M}^{k|l}_{r|s}\simeq\bigoplus_{\lambda\in\tilde{\Lambda}(k|l)\cup\tilde{\Lambda}(r|s)}L^{k|l}_{\lambda}\otimes \widetilde{L}^{r|s}_{\lambda},
\end{equation*}
where $L^{k|l}_{\lambda}$ (resp. $\widetilde{L}^{r|s}_{\lambda}$) is the irreducible left $U_\v(\mathfrak{gl}_{k|l})$-module (resp. right $U_\v(\mathfrak{gl}_{r|s})$-module) with the highest weight $\lambda$.
\end{theorem}

In order to describe the following proposition concisely, we denote 
\begin{equation*}
\tilde{a}_{ij}=\left\{
\begin{array}{ll}
a_{ij}, & \mbox{if $\hat{i}\neq\hat{j}$}\\
0, & \mbox{if $\hat{i}=\hat{j}$}
\end{array}\right.
\quad\mbox{for~} A=(a_{ij})\in M(m|n),
\end{equation*}
and introduce the quantum integer
$[a]=\frac{\v^{a}-\v^{-a}}{\v-\v^{-1}}$ for $a\in\mathbb{N}$.
\begin{proposition}\label{coor'}
Let $A=(a_{ij})\in M(k|l,r|s)$. The left $U_\v(\mathfrak{gl}_{k|l})$-action on $\mathcal{M}^{k|l}_{r|s}$ is given by
\begin{align*}
\mathop{E_{h,h+1}\cdot t^{(A)}}\limits_{(h\neq m)}&=\sum_{{j\in \tilde{I}_{s|r}}}\v_{h}^{\sum_{y>j}(a_{hy}-a_{h+1,y})}[a_{h+1,j}]t^{(A+e_{hj}-e_{h+1,j})},\\
E_{m,m+1}\cdot t^{(A)}&=\sum_{{j\in \tilde{I}_{s|r}}}(-1)^{\sum_{(x,y)<(m+1,j)}\tilde{a}_{xy}+\hat{j}}\v_{m}^{\sum_{y<j}(a_{h+1,y}-a_{h,y})}[a_{m+1,j}]t^{(A+e_{mj}-e_{m+1,j})},
\\
\mathop{E_{h+1,h}\cdot t^{(A)}}\limits_{(h\neq m)}&=\sum_{{j\in \tilde{I}_{s|r}}}\v_{h+1}^{\sum_{y<j}(a_{h+1,y}-a_{h,y})}[a_{hj}]t^{(A-e_{hj}+e_{h+1,j})},\\
E_{m+1,m}\cdot t^{(A)}&=\sum_{{j\in \tilde{I}_{s|r}}}(-1)^{\sum_{(x,y)<(m,j)}\tilde{a}_{xy}+\hat{j}+1}\v_{m+1}^{\sum_{y>j}(a_{m+1,y}+a_{my})}[a_{mj}]t^{(A-e_{mj}+e_{m+1,j})}.
\end{align*}
The right $U_\v(\mathfrak{gl}_{r|s})$-action on $\mathcal{M}^{k|l}_{r|s}$ is given by
\begin{align*}
\mathop{t^{(A)}\cdot E_{h,h+1}}\limits_{(h\neq m)}
=&\sum_{{i\in \tilde{I}_{k|l}}}(-1)^{(\hat{i}+\widehat{h+1})(\sum_{x>i}\tilde{a}_{xh}+\sum_{x<i}\tilde{a}_{x,h+1})}\v_{h}^{\sum_{x\geqslant i}(a_{xh}-a_{x,h+1})-1}[a_{ih}]t^{(A-e_{ih}+e_{i,h+1})},\\
t^{(A)}\cdot E_{m,m+1}
=&\sum_{{i\in \tilde{I}_{k|l}}}(-1)^{\sum_{(x,y)>(i,m)}\tilde{a}_{xy}+(\hat{i}+1)(\sum_{x>i}\tilde{a}_{xm}+\sum_{x<i}\tilde{a}_{x,m+1}+1)}\\
&\qquad \times \v_{m}^{\sum_{x\geqslant i}(a_{x,m+1}+a_{xm})-1}[a_{im}]t^{(A-e_{im}+e_{i,m+1})},\\
\mathop{t^{(A)}\cdot E_{h+1,h}}\limits_{(h\neq m)}
=&\sum_{{i\in \tilde{I}_{k|l}}}(-1)^{(\hat{i}+\hat{h})(\sum_{x>i}\tilde{a}_{xh}+\sum_{x<i}\tilde{a}_{x,h+1})}\v_{h+1}^{\sum_{x\leqslant i}(a_{x,h+1}-a_{xh})-1}[a_{i,h+1}]t^{(A+e_{ih}-e_{i,h+1})},\\
t^{(A)}\cdot E_{m+1,m}
=&\sum_{{i\in \tilde{I}_{k|l}}}(-1)^{\sum_{(x,y)>(i,m+1)}\tilde{a}_{xy}+\hat{i}(\sum_{x>i}\tilde{a}_{xm}+\sum_{x<i}\tilde{a}_{x,m+1}+1)}\\
&\qquad \times \v_{m+1}^{\sum_{x\leqslant i}(a_{x,m+1}+a_{xm})-1}[a_{i,m+1}]t^{(A+e_{im}-e_{i,m+1})}.
\end{align*}
\end{proposition}

\begin{proof}
Here, we present the computation only for the most two complicated cases that $h=m$. Firstly, let us calculate the left action of $E_{m,m+1}$. A straightforward calculation in (\ref{kl}) and (\ref{el}) further implies
\begin{align*}
E_{m,m+1}\cdot t_{ij}^{d}
=&\sum_{c=1}^{d}(-1)^{(c-1)(\hat{i}+\hat{j})+\hat{j}}t_{ij}^{c-1}(E_{m,m+1}\cdot t_{ij})(K_{m}K_{m+1}^{-1}\cdot t_{ij})^{d-c}\\
=&\delta_{i,m+1}\sum_{c=1}^{d}(-1)^{c(1+\hat{j})+1}\v_{m+1}^{c-d}\v_{j}^{c-1}t_{i-1,j}t_{ij}^{d-1}
=\delta_{i,m+1}(-1)^{\hat{j}}[d]t_{i-1,j}t_{ij}^{d-1}.
\end{align*}
Hence,
\begin{align*}
E_{m,m+1}\cdot t^{(A)}
=&\sum_{{j\in \tilde{I}_{s|r}}}(-1)^{\sum_{(x,y)<(m+1,j)}\tilde{a}_{xy}}(\prod\limits_{(x,y)<(m+1,j)}^{<}t_{xy}^{a_{xy}})(E_{m,m+1}\cdot t_{m+1,j}^{a_{m+1,j}})\\
&\times (\prod\limits_{(x,y)>(m+1,j)}^{<}(K_{m}K_{m+1}^{-1}\cdot t_{xy})^{a_{xy}})\\
=&\sum_{{j\in \tilde{I}_{s|r}}}(-1)^{\sum_{(x,y)>(m+1,j)}\tilde{a}_{xy}+\hat{j}}\v_{m}^{\sum_{y<j}(a_{h+1,y}-a_{h,y})}[a_{m+1,j}]t^{(A+e_{mj}-e_{m+1,j})}.
\end{align*}

Next, let us calculate the right action of $E_{m+1,m}$. Direct computation in (\ref{kr}) and (\ref{er}) further implies
\begin{align*}
t_{ij}^d\cdot E_{m+1,m}
=&\sum_{c=1}^{d}(-1)^{(d-c)(\hat{i}+\hat{j})}(t_{ij}\cdot K_{m}^{-1}K_{m+1})^{c-1}(t_{ij}\cdot E_{m+1,m})t_{ij}^{d-c}\\
=&\delta_{j,m+1}\sum_{c=1}^{d}(-1)^{(d-c)(1+\hat{i})+\hat{i}}\v_{m+1}^{c-1}\v_{i}^{c-1}t_{i,j-1}t_{ij}^{d-1}\\
=&\delta_{j,m+1}(-1)^{\hat{i}}\v_{m+1}^{d-1}[d]t_{i-1,j}t_{ij}^{d-1}.
\end{align*}
Hence,
\begin{align*}
t^{(A)}\cdot E_{m+1,m}
=&\sum_{{i\in \tilde{I}_{k|l}}}(-1)^{\sum_{(x,y)>(i,m+1)}\tilde{a}_{xy}}(\prod\limits_{(x,y)<(i,m+1)}^{<}(t_{xy} \cdot K_{m}^{-1}K_{m+1})^{a_{xy}})\\
&\times(t_{i,m+1}^{a_{i,m+1}}\cdot E_{m,m+1})(\prod\limits_{(x,y)>(i,m+1)}^{<}t_{xy}^{a_{xy}})\\
% =&\sum_{{i\in \tilde{I}_{k|l}}}(-1)^{\sum_{(x,y)>(i,m+1)}\tilde{a}_{xy}+\hat{i}(\sum_{x>i}\tilde{a}_{xm}+\sum_{x<i}\tilde{a}_{x,m+1}+1)}\\
% &\times \v_{m+1}^{\sum_{x}a_{xm}+\sum_{x<i}a_{x,m+1}}\v_{m+1}^{-\sum_{x>i}a_{xm}+a_{i,m+1}-1}[a_{i,m+1}]t_{ij}^{a_{i,m+1}-1}t^{(A+e_{im}-e_{i,m+1})},\\
=&\sum_{{i\in \tilde{I}_{k|l}}}(-1)^{\sum_{(x,y)>(i,m+1)}\tilde{a}_{xy}+\hat{i}(\sum_{x>i}\tilde{a}_{xm}+\sum_{x<i}\tilde{a}_{x,m+1}+1)}\\
&\times \v_{m+1}^{\sum_{x\leqslant i}(a_{x,m+1}+a_{xm})-1}[a_{i,m+1}]t^{(A+e_{im}-e_{i,m+1})}.
\end{align*}

The arguments for other cases are similar. 
\end{proof}

For $A=(a_{ij})\in M(m|n)$, let
\begin{equation*}
t^{\{A\}}=\v_j^{-\sum_{j\in I_{m|n}}\frac{\sum_{i}a_{ij}(\sum_{i}a_{ij}+1)}{2}}t^{(A)}.   
\end{equation*}
It is clear that $\{t^{\{A\}} \mid A\in M(k|l,r|s)\}$ is also a $\mathbb{C}(\v)$-basis of $\mathcal{M}^{k|l}_{r|s}$.

Using $t^{\{A\}}$ instead of $t^{(A)}$, we rewrite Propositions~\ref{coor'} as follows.

\begin{proposition}\label{coor}
Let $A=(a_{ij})\in M(k|l,r|s)$. The left $U_\v(\mathfrak{gl}_{k|l})$-action on $\mathcal{M}^{k|l}_{r|s}$ is given by
\begin{align*}
\mathop{E_{h,h+1}\cdot t^{\{A\}}}\limits_{(h\neq m)}&=\sum_{{j\in \tilde{I}_{s|r}}}\v_{h}^{\sum_{y>j}(a_{hy}-a_{h+1,y})}[a_{h+1,j}]t^{\{A+e_{hj}-e_{h+1,j}\}},\\
E_{m,m+1}\cdot t^{\{A\}}&=\sum_{{j\in \tilde{I}_{s|r}}}(-1)^{\sum_{(x,y)<(m+1,j)}\tilde{a}_{xy}+\hat{j}}\v_{m}^{\sum_{y<j}(a_{h+1,y}-a_{h,y})}[a_{m+1,j}]t^{\{A+e_{mj}-e_{m+1,j}\}},
\\
\mathop{E_{h+1,h}\cdot t^{\{A\}}}\limits_{(h\neq m)}&=\sum_{{j\in \tilde{I}_{s|r}}}\v_{h+1}^{\sum_{y<j}(a_{h+1,y}-a_{h,y})}[a_{hj}]t^{\{A-e_{hj}+e_{h+1,j}\}},\\
E_{m+1,m}\cdot t^{\{A\}}&=\sum_{{j\in \tilde{I}_{s|r}}}(-1)^{\sum_{(x,y)<(m,j)}\tilde{a}_{xy}+\hat{j}+1}\v_{m+1}^{\sum_{y>j}(a_{m+1,y}+a_{my})}[a_{mj}]t^{\{A-e_{mj}+e_{m+1,j}\}}.
\end{align*}
The right $U_\v(\mathfrak{gl}_{r|s})$-action on $\mathcal{M}^{k|l}_{r|s}$ is given by
\begin{align*}
\mathop{t^{\{A\}}\cdot E_{h,h+1}}\limits_{(h\neq m)}
=&\sum_{{i\in \tilde{I}_{k|l}}}(-1)^{(\hat{i}+\widehat{h+1})(\sum_{x>i}\tilde{a}_{xh}+\sum_{x<i}\tilde{a}_{x,h+1})}\v_{h}^{\sum_{x<i}(a_{x,h+1}-a_{xh})}[a_{ih}]t^{\{A-e_{ih}+e_{i,h+1}\}},\\
t^{\{A\}}\cdot E_{m,m+1}
=&\sum_{{i\in \tilde{I}_{k|l}}}(-1)^{\sum_{(x,y)>(i,m)}\tilde{a}_{xy}+(\hat{i}+1)(\sum_{x>i}\tilde{a}_{xm}+\sum_{x<i}\tilde{a}_{x,m+1}+1)}\\
&\times \v_{m}^{\sum_{x\geqslant i}(a_{xm}+a_{x,m+1})}[a_{im}]t^{\{A-e_{im}+e_{i,m+1}\}},\\
\mathop{t^{\{A\}}\cdot E_{h+1,h}}\limits_{(h\neq m)}
=&\sum_{{i\in \tilde{I}_{k|l}}}(-1)^{(\hat{i}+\hat{h})(\sum_{x>i}\tilde{a}_{xh}+\sum_{x<i}\tilde{a}_{x,h+1})}\v_{h+1}^{\sum_{x>i}(a_{xh}-a_{x,h+1})}[a_{i,h+1}]t^{\{A+e_{ih}-e_{i,h+1}\}},\\
t^{\{A\}}\cdot E_{m+1,m}
=&\sum_{{i\in \tilde{I}_{k|l}}}(-1)^{\sum_{(x,y)>(i,m+1)}\tilde{a}_{xy}+\hat{i}(\sum_{x>i}\tilde{a}_{xm}+\sum_{x<i}\tilde{a}_{x,m+1}+1)}\\
&\times \v_{m+1}^{\sum_{x\leqslant i}(a_{xm}+a_{x,m+1})}[a_{i,m+1}]t^{\{A+e_{im}-e_{i,m+1}\}}.
\end{align*}
\end{proposition}

%------------------------------------
\section{Quantum differential operators construction}
%------------------------------------

%------------------------------------
\subsection{Polynomial superalgebra}
%------------------------------------
Recall that $V=V_0\oplus V_1$ is the natural module of $U_\v(\mathfrak{gl}_{m|n})$, where $\dim V_0=m$ and $\dim V_1=n$. We have two polynomial superalgebras 
\begin{align*}
S_{m|n}:=S(V_0)\otimes \Lambda(V_1)&=\mathbb{C}(\v)[X_1,X_2,\ldots,X_{m+n}]\quad \mbox{with}\ X_i=x_i,\ X_{m+j}=\xi_j,\\
\Lambda_{m|n}:=\Lambda(V_0)\otimes S(V_1)&=\mathbb{C}(\v)[X_1,X_2,\ldots,X_{m+n}]\quad \mbox{with } X_i=\xi_i,\ X_{m+j}=x_j,
\end{align*}
where $1\leq i\leq m$ and $1\leq j\leq n$. Here, the even generators $x_i$'s are bosonic variables and the odd generators $\xi_i$'s are fermionic variables, i.e.
$$x_ix_j=x_jx_i \quad \mbox{and}\quad \xi_i\xi_j=-\xi_j\xi_i.$$ Particularly, $\xi_i^2=0$.
We use divided powers to denote their monomial bases
\begin{equation*}
X^{(\mathbf{a})}=X_1^{(a_1)}X_2^{(a_2)}\cdots X_{m+n}^{(a_{m+n})},
\end{equation*}
where $\mathbf{a}=(a_1,\ldots,a_{m+n})\in\mathbb{N}^m\times\mathbb{Z}_2^n$ for $S_{m|n}$, $\mathbf{a}=(a_1,\ldots,a_{m+n})\in\mathbb{Z}_2^m\times\mathbb{N}^n$ for $\Lambda_{m|n}$, and $X_i^{(a_i)}=\frac{X_i^{a_i}}{[a_i][a_i-1]\cdots[1]}$. The parity is defined as $\hat{\mathbf{a}}\equiv \sum_{i=1}^na_{m+i} ~ (\mathrm{mod}~2)$ for $X^{(\mathbf{a})}\in S_{m|n}$ and $\hat{\mathbf{a}}\equiv \sum_{i=1}^ma_{i} ~ (\mathrm{mod}~2)$ for $X^{(\mathbf{a})}\in \Lambda_{m|n}$.

We define the left actions of $U_\v(\mathfrak{gl}_{m|n})$ on $S_{m|n}$ and $\Lambda_{m|n}$, in the sense of quantum differential operators (cf. \cite{H90, DZ20}), saying that on $S_{m|n}$,
\begin{align*}
&K_i\cdot X^{(\mathbf{a})}=\v_i^{a_i}X^{(\mathbf{a})};\\
&E_{h,h+1}\cdot X^{(\mathbf{a})}=
\left\{
\begin{aligned}
&[a_h+1]X^{(\mathbf{a}+\mathbf{e}_h-\mathbf{e}_{h+1})},&&a_{h+1}>0,\\
&0,&&\text{otherwise};\\
\end{aligned}
\right.\\
&E_{h+1,h}\cdot X^{(\mathbf{a})}=
\left\{
\begin{aligned}
&[a_{h+1}+1]X^{(\mathbf{a}-\mathbf{e}_h+\mathbf{e}_{h+1})},&&a_h>0,\\
&0,&&\text{otherwise},\\
\end{aligned}
\right.
\end{align*}
and on $\Lambda_{m|n}$,
\begin{align*}
&K_i\cdot X^{(\mathbf{a})}=\v_i^{a_i}X^{(\mathbf{a})};\\
&E_{h,h+1}\cdot X^{(\mathbf{a})}=
\left\{
\begin{aligned}
&(-1)^{\delta_{hm}\hat{\mathbf{a}}}[a_h+1]X^{(\mathbf{a}+\mathbf{e}_h-\mathbf{e}_{h+1})},&&a_{h+1}>0,\\
&0,&&\text{otherwise};\\
\end{aligned}
\right.\\
&E_{h+1,h}\cdot X^{(\mathbf{a})}=
\left\{
\begin{aligned}
&(-1)^{\delta_{hm}\hat{\mathbf{a}}}[a_{h+1}+1]X^{(\mathbf{a}-\mathbf{e}_h+\mathbf{e}_{h+1})},&&a_h>0,\\
&0,&&\text{otherwise},\\
\end{aligned}
\right.
\end{align*}
where $1\leq i\leq m+n$ and $1\leq h<m+n$.

Consider the tensor product
\begin{equation*}
\mathcal{S}^{m|n}_{m|n}=(S_{m|n})^{\otimes m}\otimes(\Lambda_{m|n})^{\otimes n}\simeq\mathbb{C}(\v)[X_{ij}]_{1\leq i,j\leq m+n},
\end{equation*}
where $X_{ij}$ denotes the $i$-th generator of the $j$-th tensor factor. It is a polynomial superalgebra with even generators $X_{ij}$ $(\hat{i}=\hat{j})$ and odd generators $X_{ij}$ $(\hat{i}\neq\hat{j})$. It has a monomial basis
$\{X^{(A)}\mid A\in M(m|n)\}$ with
\begin{equation*}
X^{(A)}:=(-1)^{h(A)}X_{11}^{(a_{11})}X_{21}^{(a_{21})}\cdots X_{m+n,1}^{(a_{m+n,1})}\cdots X_{1,m+n}^{(a_{1,m+n})}X_{2,m+n}^{(a_{2,m+n})}\cdots X_{m+n,m+n}^{(a_{m+n,m+n})}
\end{equation*} for $A=(a_{ij})$, where we add a coefficient $h(A):=\sum_{i,j\in I_{m|n}}\frac{\tilde{a}_{ij}(\tilde{a}_{ij}+1)}{2}$ for later convenience in comparing multiplication formulas. 

The left actions of $U_\v(\mathfrak{gl}_{m|n})$ on $S_{m|n}$ and $\Lambda_{m|n}$ give a left action of $U_\v(\mathfrak{gl}_{m|n})$ on the tensor product $\mathcal{S}^{m|n}_{m|n}$. Then, by symmetry between rows and columns, we define a right action of $U_\v(\mathfrak{gl}_{m|n})$ on $\mathcal{S}^{m|n}_{m|n}$. This right action is given by the same structure coefficients as the left action, but applied to the transformed data: the monomial in $\mathcal{S}^{m|n}_{m|n}$ is reversed in order and has its indices transposed (i.e., replacing $(i,j)$ with $(j,i)$), while $x\in U_\v(\mathfrak{gl}_{m|n})$ is replaced by $\omega(x)$, where $\omega$ is the anti-involution of $U_\v(\mathfrak{gl}_{m|n})$ given in \eqref{anti}. 
Explicitly, the right action 
$$X_{m+n,m+n}^{(a_{m+n,m+n})}\cdots X_{m+n,2}^{(a_{m+n,2})}X_{m+n,1}^{(a_{m+n,1})}\cdots X_{1,m+n}^{(a_{1,m+n})}\cdots X_{12}^{(a_{12})}X_{11}^{(a_{11})}\cdot x$$ is understood via the left action $$\omega(x)\cdot X_{11}^{(a_{11})}X_{21}^{(a_{21})}\cdots X_{m+n,1}^{(a_{m+n,1})}\cdots X_{1,m+n}^{(a_{1,m+n})}X_{2,m+n}^{(a_{2,m+n})}\cdots X_{m+n,m+n}^{(a_{m+n,m+n})}.$$
A straightforward calculation shows that for $A=(a_{ij})\in M(m|n)$, 
\begin{align}\label{eq:3.1}
&X_{11}^{(a_{11})}X_{21}^{(a_{21})}\cdots X_{m+n,1}^{(a_{m+n,1})}\cdots X_{1,m+n}^{(a_{1,m+n})}X_{2,m+n}^{(a_{2,m+n})}\cdots X_{m+n,m+n}^{(a_{m+n,m+n})}\\\nonumber
=&(-1)^{s(A)}X_{m+n,m+n}^{(a_{m+n,m+n})}\cdots X_{2,m+n}^{(a_{2,m+n})}X_{1,m+n}^{(a_{1,m+n})}\cdots X_{m+n,1}^{(a_{m+n,1})}\cdots X_{21}^{(a_{21})}X_{11}^{(a_{11})}\\\nonumber
=&(-1)^{s(A)+\gamma(A)}X_{m+n,m+n}^{(a_{m+n,m+n})}\cdots X_{m+n,2}^{(a_{m+n,2})}X_{m+n,1}^{(a_{m+n,1})}\cdots X_{1,m+n}^{(a_{1,m+n})}\cdots X_{12}^{(a_{12})}X_{11}^{(a_{11})},
\end{align}
where 
$\gamma(A)=\sum_{i>k,j<l}\tilde{a}_{ij}\tilde{a}_{kl}$ and $s(A)=\sum\limits_{(i,j)<(k,l)}\tilde{a}_{ij}\tilde{a}_{kl}$. Here $(i,j)<(k,l)$ means $j<l$ or `$j=l, i<k$', as we used in \eqref{def:tA}.

%------------------------------------
\subsection{Key action formulas}
%------------------------------------

The formulas for the left action were given by Du and the third author in the following proposition.
\begin{proposition}{\cite[Lemma 3.3]{DZ20}} \label{prop:3.1}
The left action of $U_\v(\mathfrak{gl}_{m|n})$ on $\mathcal{S}_{m|n}^{m|n}$ is given by, for $A=(a_{ij})\in M(m|n)$,
\begin{align*}
\mathop{E_{h,h+1}\cdot X^{(A)}}\limits_{(h\neq m)}=&\sum_{{j\in I_{m|n}}}\v_{h}^{\sum_{y>j}(a_{hy}-a_{h+1,y})}[a_{h+1,j}]X^{(A+e_{hj}-e_{h+1,j})},\\
E_{m,m+1}\cdot X^{(A)}=&\sum_{{j\in I_{m|n}}}(-1)^{\sum_{(x,y)<(m+1,j)}\tilde{a}_{xy}+\hat{j}}\\
&\times\v_{m}^{\sum_{y<j}(a_{h+1,y}-a_{h,y})}[a_{m+1,j}]X^{(A+e_{mj}-e_{m+1,j})},\\
\mathop{E_{h+1,h}\cdot X^{(A)}}\limits_{(h\neq m)}=&\sum_{{j\in  I_{m|n}}}\v_{h+1}^{\sum_{y<j}(a_{h+1,y}-a_{h,y})}[a_{hj}]X^{(A-e_{hj}+e_{h+1,j})},\\
E_{m+1,m}\cdot X^{(A)}=&\sum_{{j\in I_{m|n}}}(-1)^{\sum_{(x,y)<(m,j)}\tilde{a}_{xy}+\hat{j}+1}\\
&\times\v_{m+1}^{\sum_{y>j}(a_{m+1,y}+a_{my})}[a_{mj}]X^{(A-e_{mj}+e_{m+1,j})}.\end{align*}
\end{proposition}

In the following, we write down explicit formulas for the right action.
\begin{proposition}\label{prop:3.2}
The right action of $U_\v(\mathfrak{gl}_{m|n})$ on $X^{(A)}\in\mathcal{S}^{m|n}_{m|n}$ for $A=(a_{ij})\in M(m|n)$ is given by
\begin{align*}
\mathop{X^{(A)}\cdot E_{h,h+1}}\limits_{(h\neq m)}
=&\sum_{{i\in I_{m|n}}}(-1)^{(\hat{i}+\widehat{h+1})(\sum_{x>i}\tilde{a}_{xh}+\sum_{x<i}\tilde{a}_{x,h+1})}\\
&\times \v_{h+1}^{\sum_{x<i}(a_{x,h+1}-a_{xh})}[a_{ih}]X^{(A-e_{ih}+e_{i,h+1})},\\
X^{(A)}\cdot E_{m,m+1}
=&\sum_{{i\in I_{m|n}}}(-1)^{\sum_{(x,y)>(i,m)}\tilde{a}_{xy}+(\hat{i}+1)(\sum_{x>i}\tilde{a}_{xm}+\sum_{x<i}\tilde{a}_{x,m+1}+1)}\\
&\times \v_{m+1}^{\sum_{x<i}(a_{x,m+1}+a_{xm})}[a_{im}]X^{(A-e_{im}+e_{i,m+1})},\\
\mathop{X^{(A)}\cdot E_{h+1,h}}\limits_{(h\neq m)}
=&\sum_{{i\in I_{m|n}}}(-1)^{(\hat{i}+\hat{h})(\sum_{x>i}\tilde{a}_{xh}+\sum_{x<i}\tilde{a}_{x,h+1})}\\
&\times \v_{h}^{\sum_{x>i}(a_{xh}-a_{x,h+1})}[a_{i,h+1}]X^{(A+e_{ih}-e_{i,h+1})},\\
X^{(A)}\cdot E_{m+1,m}
=&\sum_{{i\in I_{m|n}}}(-1)^{\sum_{(x,y)>(i,m+1)}\tilde{a}_{xy}+\hat{i}(\sum_{x>i}\tilde{a}_{xm}+\sum_{x<i}\tilde{a}_{x,m+1}+1)}\\
&\times \v_{m}^{\sum_{x>i}(a_{xm}+a_{x,m+1})}[a_{i,m+1}]X^{(A+e_{im}-e_{i,m+1})}.
\end{align*}
\end{proposition}

\begin{proof}
We only compute the right action of $E_{m,m+1}$ on $X^{(A)}$ for $A=(a_{ij})\in M(m|n)$. 

Write $A_i=A-e_{im}+e_{i,m+1}$. 
We emphasize that we need to adjust the monomial  by \eqref{eq:3.1} when we consider the right action. With this in mind, we compute 
\begin{align*}
\gamma(A_i)-\gamma(A)=&-\hat{i}(\sum_{x>i,y<m}+\sum_{x\leq m<y})a_{xy}+(\hat{i}+1)(\sum_{x>i,y>m}+\sum_{x>m\geq y})a_{xy},\\
s(A_i)-s(A)=&-\hat{i}(\sum_{\hat{x}+\hat{y}=1}a_{xy}-1)+(\hat{i}+1)\sum_{\hat{x}+\hat{y}=1}a_{xy}.
\end{align*}
Moreover, we need the following formula
\begin{align*}
\sum_{(y,x)<(m,i)}\tilde{a}_{yx}=\hat{i}(\sum_{x\leq m<y}\tilde{a}_{xy}+\sum_{x\leq i,y\leq m}\tilde{a}_{xy}-1)+(\hat{i}+1)\sum_{x<i,y>m}a_{xy}.  
\end{align*}
Hence, according to the expression for the left action given in Proposition~\ref{prop:3.1}, we obtain
\begin{align*}
&X^{(A)}\cdot E_{m,m+1}\\
=&\sum_{{i\in I_{m|n}}}(-1)^{s(A)+\gamma(A)+s(A_i)+\gamma(A_i)}(-1)^{\sum_{(y,x)<(m,i)}\tilde{a}_{yx}+\hat{i}+1}\v_{m+1}^{\sum_{y<i}a_{y,m+1}+a_{ym}}[a_{im}]X^{(A_i)}\\
=&\sum_{{i\in I_{m|n}}}(-1)^{\sum_{(x,y)>(i,m)}\tilde{a}_{xy}+(\hat{i}+1)(\sum_{x>i}\tilde{a}_{xm}+\sum_{x<i}\tilde{a}_{x,m+1}+1)}\v_{m+1}^{\sum_{x<i}(a_{x,m+1}+a_{xm})}[a_{im}]X^{(A_i)}. 
\end{align*}
The arguments for other cases are similar.
\end{proof}

Recall $\mathcal{M}^{k|l}_{r|s}$ studied in Section \ref{sec:2}. Note that $\mathcal{M}^{m|n}_{m|n}$ and $\mathcal{S}^{m|n}_{m|n}$ each have a basis indexed by the same set $M(m|n)$, and therefore isomorphic as $\mathbb{C}(\v)$-spaces. Comparing Propositions \ref{prop:3.1} and \ref{prop:3.2} with Proposition~\ref{coor} at specialization $k=r=m, l=s=n$, we see $\mathcal{M}^{m|n}_{m|n}\simeq \mathcal{S}^{m|n}_{m|n}$ as $U_\v(\mathfrak{gl}_{m|n})$-bimodules. In particular, the left and right actions of $U_\v(\mathfrak{gl}_{m|n})$ on $\mathcal{S}^{m|n}_{m|n}$ commute with each other.

Let $k,r\leq m$, and $l,s\leq n$. We consider $S_{k|l}$ (resp. $\Lambda_{k|l}$) as a submodule of $S_{m|n}$ (resp. $\Lambda_{m|n}$) in the obvious way. Consider the tensor product
\begin{equation}\label{eq:S}
\mathcal{S}^{k|l}_{r|s}=(S_{k|l})^{\otimes r}\otimes(\Lambda_{k|l})^{\otimes s}.
\end{equation}
The basis element $X^{(\mathbf{a})}$ of $S_{k|l}$ or $\Lambda_{k|l}$ is in the form of 
\begin{equation*}
X^{(\mathbf{a})}=X_1^{(0)}\cdots X_{m-k+1}^{(a_{m-k+1})}\cdots X_{m+l}^{(a_{m+l})}\cdots X_{m+n}^{(0)}
\end{equation*}
and the basis element $X^{(A)}$ of $\mathcal{bS}^{k|l}_{r|s}$ is in the form of 
\begin{equation*}
X^{(A)}=X^{(\mathbf{0})}\cdots X^{(\mathbf{c}_{m-r+1})}\cdots X^{(\mathbf{c}_{m+s})}\cdots X^{(\mathbf{0})}.
\end{equation*}
In other words, the tensor product $\mathcal{S}^{k|l}_{r|s}$ has a basis 
$\{X^{(A)} \mid A\in M(k|l,r|s)\}$.

Regarding $\mathcal{S}^{k|l}_{r|s}$ as a submodule of $\mathcal{S}^{m|n}_{m|n}$, we immediately obtain the following propositions by Propositions \ref{prop:3.1} \& \ref{prop:3.2}.
\begin{proposition} \label{prop:3.3}
\item[(1)] The left action of $U_\v(\mathfrak{gl}_{k|l})$ on $X^{(A)}\in\mathcal{S}^{k|l}_{r|s}$ is as follows: for $h\in\tilde{I}_{r|s}$,
\begin{align*}
\mathop{E_{h,h+1}\cdot X^{(A)}}\limits_{(h\neq m)}=&\sum_{{j\in \tilde{I}_{s|r}}}\v_{h}^{\sum_{y>j}(a_{hy}-a_{h+1,y})}[a_{h+1,j}]X^{(A+e_{hj}-e_{h+1,j})},\\
E_{m,m+1}\cdot X^{(A)}=&\sum_{{j\in \tilde{I}_{s|r}}}(-1)^{\sum_{(x,y)>(m+1,j)}\tilde{a}_{xy}}\v_{m}^{\sum_{y<j}(a_{h+1,y}-a_{h,y})}[a_{m+1,j}]X^{(A+e_{mj}-e_{m+1,j})},\\
\mathop{E_{h+1,h}\cdot X^{(A)}}\limits_{(h\neq m)}=&\sum_{{j\in \tilde{I}_{s|r}}}\v_{h+1}^{\sum_{y<j}(a_{h+1,y}-a_{h,y})}[a_{hj}]X^{(A-e_{hj}+e_{h+1,j})},\\
E_{m+1,m}\cdot X^{(A)}=&\sum_{{j\in \tilde{I}_{s|r}}}(-1)^{\sum_{(x,y)>(m,j)}\tilde{a}_{xy}}\v_{m+1}^{\sum_{y>j}(a_{m+1,y}+a_{my})}[a_{mj}]X^{(A-e_{mj}+e_{m+1,j})}.\end{align*}

\item[(2)] The right action of $U_\v(\mathfrak{gl}_{r|s})$ on $X^{(A)}\in\mathcal{S}^{k|l}_{r|s}$ is as follows: for $h\in\tilde{I}_{k|l}$,
\begin{align*}
\mathop{X^{(A)}\cdot E_{h,h+1}}\limits_{(h\neq m)}
=&\sum_{{i\in \tilde{I}_{k|l}}}(-1)^{(\hat{i}+\widehat{h+1})(\sum_{x>i}\tilde{a}_{xh}+\sum_{x<i}\tilde{a}_{x,h+1})}\\
&\times \v_{h+1}^{\sum_{x<i}(a_{x,h+1}-a_{xh})}[a_{ih}]X^{(A-e_{ih}+e_{i,h+1})},\\
X^{(A)}\cdot E_{m,m+1}
=&\sum_{{i\in \tilde{I}_{k|l}}}(-1)^{\sum_{(x,y)<(i,m)}\tilde{a}_{xy}+(\hat{i}+1)(\sum_{x>i}\tilde{a}_{xm}+\sum_{x<i}\tilde{a}_{x,m+1})}\\
&\times \v_{m+1}^{\sum_{x<i}(a_{x,m+1}+a_{xm})}[a_{im}]X^{(A-e_{im}+e_{i,m+1})},\\
\mathop{X^{(A)}\cdot E_{h+1,h}}\limits_{(h\neq m)}
=&\sum_{{i\in \tilde{I}_{k|l}}}(-1)^{(\hat{i}+\hat{h})(\sum_{x>i}\tilde{a}_{xh}+\sum_{x<i}\tilde{a}_{x,h+1})}\\
&\times \v_{h}^{\sum_{x>i}(a_{xh}-a_{x,h+1})}[a_{i,h+1}]X^{(A+e_{ih}-e_{i,h+1})},\\
X^{(A)}\cdot E_{m+1,m}
=&\sum_{{i\in \tilde{I}_{k|l}}}(-1)^{\sum_{(x,y)<(i,m+1)}\tilde{a}_{xy}+\hat{i}(\sum_{x>i}\tilde{a}_{xm}+\sum_{x<i}\tilde{a}_{x,m+1})}\\
&\times \v_{m}^{\sum_{x>i}(a_{xm}+a_{x,m+1})}[a_{i,m+1}]X^{(A+e_{im}-e_{i,m+1})}.
\end{align*}
\end{proposition}

%=====================================================
\section{Beilinson-Lusztig-MacPherson construction}
%=====================================================

%------------------------------------
\subsection{Symmetric group and parabolic subgroup}
%------------------------------------

Let $W=\mathfrak{S}_d$ be the symmetric group with simple reflections $S=\{s_i=(i,i+1) \mid 1\leq i<d\}$ and unity $\mathbbm{1}$. Denote by $\ell: W\rightarrow\mathbb{N}$ the length function with respect to $S$.

Let
\begin{equation*}
\Lambda(m+n,d)=\{\lambda=(\lambda_1,\lambda_2,\ldots,\lambda_{m+n}) \in\mathbb{N}^{m+n} \mid \lambda_1+\lambda_2+\ldots+\lambda_{m+n}=d\}
\end{equation*} be the set of compositions of $d$ into $m+n$ parts. 
%For $\lambda\in\Lambda(m+n,d)$, write $\lambda^{(0)}=(\lambda_1,\ldots,\lambda_{m})$ and $\lambda^{(1)}=(\lambda_{m+1},\ldots,\lambda_{m+n})$.
%
The parabolic subgroup $W_{\lambda}=W_{\lambda}^1\times\cdots\times W_{\lambda}^{m+n}$ associated with $\lambda\in\Lambda(m+n,d)$ is defined as the subgroup of $W$ generated by $S\backslash\{s_{\tilde{\lambda}_1},s_{\tilde{\lambda}_2},\ldots,s_{\tilde{\lambda}_{m+n-1}}\}$, where $W_{\lambda}^i$ keeps the following set $R^{\lambda}_i$ invariant: \begin{equation*}
R^{\lambda}_i=\{\tilde{\lambda}_{i-1}+1,\tilde{\lambda}_{i-1}+2,\ldots,\tilde{\lambda}_{i-1}+\lambda_i\},\quad 1\leq i\leq m+n.
\end{equation*}
Here $\tilde{\lambda}_0=0$ and $\tilde{\lambda}_i=\lambda_1+\lambda_2+\cdots+\lambda_i$, ($1\leq i\leq m+n$).
Write $W_{\lambda^{(0)}}:=W_\lambda^1\times\cdots\times W_\lambda^m$ and $W_{\lambda^{(1)}}:=W_\lambda^{m+1}\times\cdots\times W_\lambda^{m+n}$.

For $\lambda,\rho\in\Lambda(m+n,d)$, let $\mathcal{D}_{\lambda}$ be the set of all minimal length representatives of the right cosets $W_{\lambda}\backslash W$, and let $\mathcal{D}_{\lambda\rho}=\mathcal{D}_{\lambda}\cap\mathcal{D}^{-1}_{\rho}$ be the set of all minimal length representatives of the double cosets $W_{\lambda}\backslash W/W_{\rho}$.
For $g\in\mathcal{D}_{\lambda\rho}$, let $\rho g^{-1}\cap\lambda$ denote the composition such that
\begin{equation*}
W_{\lambda\cap\rho g^{-1}}:=W_{\lambda}\cap gW_{\rho}g^{-1},
\end{equation*}
which decomposes into four parabolic subgroups
\begin{equation*}
W_{\lambda\cap\rho g^{-1}}=(W^{g}_{\lambda^{(0)}}\cap W_{\rho^{(0)}})\times (W^{g}_{\lambda^{(1)}}\cap W_{\rho^{(1)}})\times (W^{g}_{\lambda^{(0)}}\cap W_{\rho^{(1)}})\times (W^{g}_{\lambda^{(1)}}\cap W_{\rho^{(0)}}),
\end{equation*}
where $W^{g}_{\lambda^{(i)}}:=g W_{\lambda^{(i)}} g^{-1}$.
We say that $g$ satisfies the even-odd trivial intersection property if $W^{g}_{\lambda^{(0)}}\cap W_{\rho^{(1)}}=W^{g}_{\lambda^{(1)}}\cap W_{\rho^{(0)}}=\{\mathbbm{1}\}$. 
Denote
\begin{equation*}
\mathcal{D}^{\circ}_{\lambda\rho}=\{g\in\mathcal{D}_{\rho\lambda} \mid W^{g}_{\lambda^{(0)}}\cap W_{\rho^{(1)}}=W^{g}_{\lambda^{(1)}}\cap W_{\rho^{(0)}}=\{\mathbbm{1}\}\}.
\end{equation*}

The composition $\lambda \cap\rho g^{-1}$ can be described in terms of the matrix
\begin{equation*}
\jmath(\lambda,g,\rho):=(a_{ij})\in \mathrm{Mat}_{m+n}(\mathbb{N})\quad \mbox{with}\quad a_{ij}=|R^{\lambda}_i\cap g^{-1}(R^{\rho}_j)|,
\end{equation*}
saying
\begin{equation*}
\lambda \cap\rho g^{-1}=(a_{11},a_{12},\ldots,a_{1,m+n},a_{21},a_{22},\ldots, a_{2,m+n},\ldots,a_{m+n,1},a_{m+n,2},\ldots,a_{m+n,m+n}).
\end{equation*}
Denote
\begin{equation*}
M(m|n,d)=\{\jmath(\lambda,g,\rho) \mid \lambda,\rho\in\Lambda(m+n,d),g\in\mathcal{D}^{\circ}_{\lambda\rho}\},
\end{equation*}
which is a subset of $M(m|n)$.
For $A=\jmath(\lambda,g,\rho)\in M(m|n,d)$, we denote $\mathrm{ro}(A)=\lambda$ and $\mathrm{co}(A)=\rho$. 

%------------------------------------
\subsection{Hecke algebra}
%------------------------------------
The Hecke algebra $\mathbf{H}=\mathbf{H}_{\v}(W)$ is spanned by a $\mathbb{C}(\v)$-basis $\{T_{w}|w\in W\}$, whose multiplication is defined by, for $s\in S$ and $w\in W$,
\begin{equation*}
T_{w}T_{s}=\left\{
\begin{aligned}
&T_{ws},&&\ell(ws)>\ell(w);\\
&(\v^2-1)T_{w}+\v^2 T_{ws},&&\text{otherwise}.\\
\end{aligned}
\right.
\end{equation*}

Let $V(m|n)$ be a $(m+n)$-dimensional $\mathbb{C}(\v)$-space with a basis $v_1,v_2,\ldots,v_{m+n}$. Its tensor product $V(m|n)^{\otimes d}$ has a basis $\{v_{\i}|\i\in I(m|n,d)\}$, where
\begin{equation*}
I(m|n,d):=\{\i=(i_1,i_2,\cdots,i_d)|1\leq i_j\leq m+n\}, \quad \mbox{and}
\end{equation*}
\begin{equation*}
v_{\i}=v_{i_1}\otimes v_{i_2}\otimes\cdots\otimes v_{i_d}=v_{i_1}v_{i_2}\cdots v_{i_d}.\end{equation*}
There is a right $W$-action on $I(m|n,d)$ as follows: for $w\in W$, $\i\in I(m|n,d)$,
\begin{equation*}
\i w=(i_{w^{-1}(1)},i_{w^{-1}(2)},\cdots,i_{w^{-1}(d)}).
\end{equation*}

For $\lambda\in\Lambda(m+n,d)$, define $\i_{\lambda}\in\Lambda(m+n,d)$ by
\begin{equation*}\i_{\lambda}=(\underbrace{1,\ldots ,1}_{\lambda_1}, \underbrace{2,\ldots ,2}_{\lambda_2},\ldots, \underbrace {m+n,\ldots ,m+n}_{\lambda_{m+n}}\}=(1^{\lambda_1},2^{\lambda_2},\ldots,(m+n)^{\lambda_{m+n}}).
\end{equation*} We rewrite $$v_\lambda:=v_{\i_\lambda}=\underbrace{v_1\otimes\cdots\otimes v_1}_{\lambda_1}\otimes \underbrace{v_2\otimes\cdots\otimes v_2}_{\lambda_2}\otimes \cdots\otimes \underbrace{v_{m+n}\otimes\cdots\otimes v_{m+n}}_{\lambda_{m+n}}.$$ 

\begin{lemma}{\cite[(1.0.10)]{DG14}} The tensor space
$V(m|n)^{\otimes d}$ is a right $\mathbf{H}$-module via
\begin{equation*}
v_{\i}T_{s_k}=\left\{
\begin{aligned}
&(-1)^{\hat{i}_k\hat{i}_{k+1}}v_{\i{s_k}},&&i_k<i_{k+1};\\
&\v^2 v_{\i},&&1\leq i_k=i_{k+1}\leq m;\\
&-v_{\i},&&m+1\leq i_k=i_{k+1};\\
&(-1)^{\hat{i}_k\hat{i}_{k+1}}\v^2 v_{\i{s_k}}+(\v^2-1)v_{\i},&&i_k>i_{k+1}.\\
\end{aligned}
\right.
\end{equation*}
Moreover, $V(m|n)^{\otimes d}\simeq\mathop{\bigoplus}\limits_{\lambda\in I(m|n,d)}v_{\lambda}\mathbf{H}$ as right $\mathbf{H}$-modules.
\end{lemma}

%---------------------
\subsection{Quantum Schur superalgebra}
%---------------------
Let
\begin{equation*}
\mathbf{S}=\mathbf{S_{\v}}(m|n,d):=\mathrm{End}_{\mathbf{H}}(V(m|n)^{\otimes d})
\end{equation*}
be the quantum Schur superalgebra over $\mathbb{C}(\v)$ with a $\mathbb{Z}_2$-grading
\begin{equation*}
\mathbf{S}_{i}:=\bigoplus_{|\lambda^{(1)}|+|\mu^{(1)}| \equiv i\ (\text{mod}\ 2)}\mathrm{Hom}_{\mathbf{H}}(v_{\lambda}\mathbf{H},v_{\mu}\mathbf{H}),\quad i\in \mathbb{Z}_2.
\end{equation*}

For $\i,\j\in I(m|n,d)$, let $e_{\i\j}\in\text{End}_{\mathbb{C}(\v)}(V(m|n)^{\otimes d})$ be the $\mathbb{C}(\v)$-linear map determined by
\begin{equation*}
(v_{\i'})e_{\i\j}=
\left\{
\begin{aligned}
&v_{\j},&\mbox{if } \i'=\i;\\
&0,&\mbox{if } \i'\neq\i.
\end{aligned}
\right.
\end{equation*}
For $A=\jmath(\lambda,g,\rho)\in M(m|n,d)$, write
\begin{equation*}
e_{A}:=\sum_{w\in\mathcal{D}_{\lambda \cap\rho g^{-1}}}\v^{-2\ell(w)}T_{w^{-1}}e_{\i_\lambda,\i_{\rho g^{-1}}}T_{w}\in\text{End}_{\mathbb{C}(\v)}(V(m|n)^{\otimes d}).
\end{equation*}

It was shown in \cite[Lemma 1.2]{DG14} that the set $\{e_{A}\mid A\in M(m|n,d)\}$ forms a $\mathbb{C}(\v)$-basis of $\mathbf{S}$. Furthermore, $\mathbf{S}$ is a quotient of $U_\v(\mathfrak{gl}_{m|n})$, i.e., there is a surjective homomorphism $\varphi: U_\v(\mathfrak{gl}_{m|n})\to \mathbf{S}$.

%------------------------------------
\subsection{Key multiplication formuals}
%------------------------------------
Let $e_{ij}\in M(m|n)$ be the matrix whose $(i,j)$-th entry is $1$ and the others are $0$, $(i,j\in I_{m|n})$.
Denote $\llbracket a\rrbracket=\frac{\v^{2a}-1}{\v^2-1}$ for $a\in\mathbb{N}$.

\begin{proposition}{\cite[Lemma 3.1]{DG14}}\label{right}
If $A,B,C\in M(m|n,d)$, $h\in I_{m|n}/\{m+n\}$ satisfy that $\mathrm{co}(A)=\mathrm{ro}(B)=\mathrm{ro}(C)$ and $B-e_{h,h+1}$ and $C-e_{h+1,h}$ are diagonal, then the following multiplication formulas hold: 
\begin{align*}
e_{A}e_{B}&=\sum_{a_{ih}\geq1} f_i(\v;A,h+1)\llbracket a_{i,h+1}+1\rrbracket_{\v_{h+1}}e_{A-e_{ih}+e_{i,h+1}},  
\\
e_{A}e_{C}&=\sum_{a_{i,h}\geq 1}g_i(\v;A,h)\llbracket a_{ih}+1\rrbracket_{\v_h}e_{A+e_{ih}-e_{i,h+1}},    
\end{align*}
where $e_{A\pm e_{ih}\mp e_{i,h+1}}=0$ whenever $A\pm e_{ih}\mp e_{i,h+1}\notin M(m|n,d)$, and
\begin{align*}
&f_i(\v;A,h+1)=
\left\{
\begin{aligned}
&\v^{2\sum_{x<i}a_{x,h+1}},&&h<m;\\
&(-1)^{\sum_{\scalebox{0.7}{$\substack{x<i\\y>m}$}}a_{xy}},&&h=m;\\
&\v^{-2\sum_{x>i}a_{xh}},&&h>m,
\end{aligned}
\right.\\
&g_i(\v;A,h)=
\left\{
\begin{aligned}
&\v^{2\sum_{x>i}a_{xh}},&&h<m;\\
&(-1)^{\sum_{\scalebox{0.7}{$\substack{x<i\\y>m}$}}a_{xy}}\v^{-2\sum_{x<i}a_{xm}}\v^{2\sum_{x>i}a_{xm}},&&h=m;\\
&\v^{-2\sum_{x<i}a_{x,h+1}},&&h>m.
\end{aligned}
\right.
\end{align*}
\end{proposition}

Let $\tau$ be the anti-involution of $\mathbf{S}$ induced from the anti-involution $\omega$ of $U_\v(\mathfrak{gl}_{m|n})$ via the surjective homomorphism $\phi: U_\v(\mathfrak{gl}_{m|n})\to \mathbf{S}$, so that $\tau\circ\varphi=\varphi\circ\omega$. 
By \cite[Remark 5.1(2)]{DG15}, there exists a basis $\{\phi_A\}_{A\in M(m|n,d)}$ of $\mathbf{S}$ such that $\tau(\phi_A)=\phi_{A^T}$.
Our basis element $e_{A}$ is denoted by $N_{A^T}$ in \cite{DG14}. Comparing $e_A=N_{A^T}$ with $\phi_A$ therein, we have 
\begin{equation}\label{eq:taueA}
\tau(e_A)=(-1)^{\widehat{A}+\widehat{A^T}}e_{A^T},
\end{equation}
where 
$\widehat{A}:=\sum_{\substack{i>k>m\\j<l}} a_{ij}a_{kl}$ for $A=(a_{ij})\in M(m|n,d)$.

\begin{proposition}\label{left}
Let $A,B,C\in M(m|n,d)$, $h\in I_{m|n}/\{m+n\}$ be such that $\mathrm{ro}(A)=\mathrm{co}(B)=\mathrm{co}(C)$ and $B-e_{h,h+1}$ and $C-e_{h+1,h}$ are diagonal. The following multiplication formulas hold: 
\begin{align*}
e_{B}e_{A}&=\sum_{a_{h+1,j}\geq1}f'_j(\v;A,h)\llbracket a_{hj}+1\rrbracket_{\v_h}e_{A+e_{hj}-e_{h+1,j}},    
\\
e_{C}e_{A}&=\sum_{a_{hj}\geq1}g'_j(\v;A,h+1)\llbracket a_{h+1,j}+1\rrbracket_{\v_{h+1}}e_{A-e_{hj}+e_{h+1,j}}, 
\end{align*}
where $e_{A\pm e_{hj}\mp e_{h+1,j}}=0$ whenever $A\pm e_{hj}\mp e_{h+1,j}\notin M(m|n,d)$, and

\begin{equation*}
f'_j(\v;A,h)=\left\{
\begin{aligned}
&(-1)^{\sum_{y>j}\hat{j}\hat{y}a_{hy}+\sum_{y<j}\hat{j}\hat{y}a_{h+1,y}}\v^{2\sum_{y>j}a_{hy}},&&h<m;\\
&(-1)^{\sum_{y>j}\hat{j}\hat{y}a_{my}+\sum_{y<j}(1+\hat{j}\hat{y})a_{m+1,y}}\v^{-2\sum_{y<j}a_{m+1,y}}\v^{2\sum_{y>j}a_{my}},&&h=m;\\
&(-1)^{\sum_{y>j}(1+\hat{j}\hat{y})a_{hy}+\sum_{y<j}(1+\hat{j}\hat{y})a_{h+1,y}}\v^{-2\sum_{y<j}a_{h+1,y}},&&h>m,
\end{aligned}
\right.
\end{equation*}
\begin{equation*}
g'_j(\v;A,h+1)=\left\{
\begin{aligned}
&(-1)^{\sum_{y>j}\hat{j}\hat{y}a_{hy}+\sum_{y<j}\hat{j}\hat{y}a_{h+1,y}}\v^{2\sum_{y<j}a_{h+1,y}},&&h<m;\\
&(-1)^{\sum_{y>j}\hat{j}\hat{y}a_{my}+\sum_{y<j}(1+\hat{j}\hat{y})a_{m+1,y}},&&h=m;\\
&(-1)^{\sum_{y>j}(1+\hat{j}\hat{y})a_{hy}+\sum_{y<j}(1+\hat{j}\hat{y})a_{h+1,y}}\v^{-2\sum_{y>j}a_{hy}},&&h>m.
\end{aligned}
\right.
\end{equation*}
\end{proposition}

\begin{proof}
In fact, this proposition is essentially a reformulation of \cite[Theorem~4.1]{DGZ18}, in which the formulas are given in terms of the basis $\{\phi_A\}$. So one can prove the proposition directly by finding the relationship between the basis $\{e_A\}$ and the basis $\{\phi_A\}$. 
To avoid introducing the notation from \cite{DGZ18} in detail, here we provide an alternative argument based on Proposition~\ref{left} by using \eqref{eq:taueA} and 
$e_{B}e_{A}=\tau(\tau(e_A)\tau(e_B))$. This method was already used by Du, Gu and the third author to reprove \cite[Lemma~3.1]{DG14} through \cite[Theorem~4.1 \& Corollary~4.3]{DGZ18}.

Consider, for example, the case where $B-e_{m,m+1}$ is diagonal. Write $A_j=A+e_{mj}-e_{m+1,j}\in M(m|n,d)$. Note that 
\begin{equation*}
\widehat{A_j^{T}}-\widehat{A^{T}}=\sum_{x>m,y<j}a_{xy}\quad \text{and}\quad
\widehat{A_j}-\widehat{A}=\hat{j}(\sum_{y<j}a_{my}-\sum_{j<y<m+1}a_{m+1,y}).
\end{equation*}
Hence,
\begin{align*}
e_{B}e_{A}=&(-1)^{\widehat{A^{T}}+\widehat{A}}\tau(e_{A^T}e_{B^T})\\
=&\sum_{{j\in I_{m|n}}}(-1)^{\widehat{A^{T}}+\widehat{A}+\widehat{A_j^{T}}+\widehat{A_j}}(-1)^{\sum_{\scalebox{0.7}{$\substack{x>m\\y<j}$}}a_{xy}}\v^{-2\sum_{y<j}a_{my}}\v^{2\sum_{y>j}a_{my}}\llbracket a_{mj}+1\rrbracket_{\v_{m}}e_{A_j}\\
=&\sum_{{j\in I_{m|n}}}(-1)^{\sum_{y>j}\hat{j}\hat{y}a_{my}+\sum_{y<j}(1+\hat{j}\hat{y})a_{m+1,y}}\v^{-2\sum_{y<j}a_{m+1,y}}\v^{2\sum_{y>j}a_{my}}\llbracket a_{mj}+1\rrbracket_{\v_{m}}e_{A_j}. 
\end{align*}

The proof for other cases is similar.
\end{proof}

For $A=(a_{ij})\in M(m|n,d)$, let
$$[A]=(-1)^{d'(A)+h(A)}\v^{-2d(A)}\frac{\tau(e_{A^T})}{\prod_{(i,j)}([a_{ij}][a_{ij}-1]\cdots[1])},\quad\mbox{where}$$
\begin{gather*}
d(A)=\sum_{\substack{i>k\\j<l}}a_{ij}a_{kl}+\sum_{j<l}(-1)^i a_{ij}a_{il}, \quad \mbox{and}\\ 
d'(A)=\sum_{\substack{i>k>m\\j<m,j<l}}a_{ij}a_{kl}+\sum_{\substack{i<k<m\\j>l>m}}a_{ij}a_{kl}+\frac{(\sum_{i\leq m<j} a_{ij}-1)\sum_{i\leq m<j} a_{ij}}{2}.
\end{gather*} It is clear that $\{[A] \mid A\in M(m|n,d)\}$ is also a $\mathbb{C}(\v)$-basis of $\mathbf{S}$.

Using $[A]$ instead of $e_A$, we rewrite Propositions~\ref{right} \& \ref{left} as follows.
\begin{proposition}\label{3.6}
For $A,B,C\in M(m|n,d),~h\in[1,m+n)$, the following multiplication formulas hold.
\begin{itemize}
\item[(1)] If $B-e_{h,h+1}$ is diagonal, $\mathrm{co}(A)=\mathrm{ro}(B)$, then
\begin{align*}
\mathop{[A][B]}\limits_{(h\neq m)}
=&\sum_{\substack{i\in I_{m|n}\\a_{ih}\geq1}}(-1)^{(\hat{i}+\widehat{h+1})(\sum_{x>i}\tilde{a}_{xh}+\sum_{x<i}\tilde{a}_{x,h+1})}\\
&\times \v_{h+1}^{\sum_{x<i}(a_{x,h+1}-a_{xh})}[a_{ih}][A-e_{ih}+e_{i,h+1}],\\
\mathop{[A][B]}\limits_{(h=m)}
=&\sum_{\substack{i\in I_{m|n}\\a_{ih}\geq1}}(-1)^{\sum_{(x,y)>(i,m)}\tilde{a}_{xy}+(\hat{i}+1)(\sum_{x>i}\tilde{a}_{xm}+\sum_{x<i}\tilde{a}_{x,m+1}+1)}\\
&\times \v_{m+1}^{\sum_{x<i}(a_{x,m+1}+a_{xm})}[a_{im}][A-e_{im}+e_{i,m+1}].
\end{align*}
\item[(2)] If $C-e_{h+1,h}$ is diagonal, $\mathrm{co}(A)=\mathrm{ro}(C)$, then
\begin{align*}
\mathop{[A][C]}_{(h\neq m)}
=&\sum_{\substack{i\in I_{m|n}\\a_{i,h+1}\geq1}}(-1)^{(\hat{i}+\hat{h})(\sum_{x>i}\tilde{a}_{xh}+\sum_{x<i}\tilde{a}_{x,h+1})}\\
&\times \v_{h}^{\sum_{x>i}(a_{xh}-a_{x,h+1})}[a_{i,h+1}][A+e_{ih}-e_{i,h+1}],\\
\mathop{[A][C]}_{(h=m)}
=&\sum_{\substack{i\in I_{m|n}\\a_{i,m+1}\geq1}}(-1)^{\sum_{(x,y)>(i,m+1)}\tilde{a}_{xy}+\hat{i}(\sum_{x>i}\tilde{a}_{xm}+\sum_{x<i}\tilde{a}_{x,m+1}+1)}\\
&\times \v_{m}^{\sum_{x>i}a_{xm}+a_{x,m}}[a_{i,m+1}][A+e_{im}-e_{i,m+1}].
\end{align*}
\end{itemize}
\end{proposition}

\begin{proposition}\label{standard}
For $A,B,C\in M(m|n,d),~h\in[1,m+n)$, the following multiplication formulas hold.
\begin{itemize}
\item[(1)] If $B-e_{h,h+1}$ is diagonal, $\mathrm{co}(B)=\mathrm{ro}(A)$, then
\begin{align*}
\mathop{[B][A]}\limits_{(h\neq m)}=&\sum_{\substack{j\in I_{m|n}\\a_{h+1,j}\geq1}}\v_{h}^{\sum_{y>j}a_{hy}-{a_{h+1,y}}}[a_{h+1,j}][A+e_{hj}-e_{h+1,j}],\\
\mathop{[B][A]}\limits_{(h=m)}=&\sum_{\substack{j\in I_{m|n}\\a_{m+1,j}\geq1}}(-1)^{\sum_{(x,y)<(m+1,j)}\tilde{a}_{xy}+\hat{j}}\v_{m}^{\sum_{y>j}a_{my}+{a_{m+1,y}}}[a_{m+1,j}][A+e_{mj}-e_{m+1,j}].
\end{align*}
\item[(2)] If $C-e_{h+1,h}$ is diagonal, $\mathrm{co}(C)=\mathrm{ro}(A)$, then
\begin{align*}
\mathop{[C][A]}\limits_{(h\neq m)}=&\sum_{\substack{j\in I_{m|n}\\a_{hj}\geq1}}\v_{h+1}^{\sum_{y<j}a_{h+1,y}-{a_{hy}}}[a_{hj}][A-e_{hj}+e_{h+1,j}],\\
\mathop{[C][A]}\limits_{(h=m)}=&\sum_{\substack{j\in I_{m|n}\\a_{mj}\geq1}}(-1)^{\sum_{(x,y)<(m,j)}\tilde{a}_{xy}+\hat{j}+1}\v_{m+1}^{\sum_{y<j}a_{m+1,y}+{a_{my}}}[a_{mj}][A-e_{mj}+e_{m+1,j}].
\end{align*}
\end{itemize}
\end{proposition}

%---------------------
\subsection{The $(\mathbf{S}_{\v}(k|l,d),\mathbf{S}_{\v}(r|s,d))$-moudle $\mathcal{V}^{k|l}_{r|s}(d)$} \label{sev:V}
%---------------------

Denote
\begin{equation*}
M(k|l,r|s;d)=\{\jmath(\lambda,g,\rho) \mid \rho\in\tilde{\Lambda}(k|l,d),~\lambda\in\tilde{\Lambda}(r|s,d),~g\in\mathcal{D}^{\circ}_{\lambda\rho}\},
\end{equation*}
where
\begin{equation*}
\tilde{\Lambda}(a+b,d)=\{\lambda=(0,\ldots,0,\lambda_{m-a+1},\ldots,\lambda_{m},\lambda_{m+1},\ldots,\lambda_{m+b},0,\ldots,0)\in\Lambda(m+n,d)\}.
\end{equation*}
Note that $M(k|l,r|s)=\bigcup_{d\geq 1}M(k|l,r|s;d)$.

For $\lambda\neq\mu\in{\Lambda(m+n,d)}$, note that $[\mathrm{diag}(\lambda)][\mathrm{diag}(\mu)]=0$ and $[\mathrm{diag}(\lambda)]^2=[\mathrm{diag}(\lambda)]$. That is, $\{[\mathrm{diag}(\lambda)]\mid \lambda\in\Lambda(m+n,d)\}$ forms a set of orthogonal idempotent elements. 

Denote $\xi_{k|l}:=\sum_{\lambda\in\tilde{\Lambda}(k|l,d)}[\mathrm{diag}(\lambda)]$. We have $$\mathbf{S_{\v}}(k|l,d)\simeq\xi_{k|l}\mathbf{S_{\v}}(m|n,d)\xi_{k|l}.$$ Moreover, let $$\mathcal{V}^{k|l}_{r|s}(d):=\xi_{k|l}\mathbf{S_{\v}}(m|n,d)\xi_{r|s},$$ 
which has a $\mathbb{C}(\v)$-basis $\{[A] \mid A\in M(k|l,r|s;d)\}$. It is clear that $\mathcal{V}^{k|l}_{r|s}(d)$ admits a left $\mathbf{S}_{\v}(k|l,d)$-action and a right $\mathbf{S}_{\v}(r|s,d)$-action. These two actions commute naturally, as they are both induced by the multiplication of the larger algebra $\mathbf{S_{\v}}(m|n,d)$.

\begin{proposition}\label{3.8}
Let $A=(a_{ij})\in M(k|l,r|s;d)$ and $B,C\in M(r|s,d)$ satisfy that $\mathrm{co}(A)=\mathrm{ro}(B)=\mathrm{ro}(C)$ and $B-e_{h,h+1}$ and $C-e_{h+1,h}$ are diagonal for some $h\in\tilde{I}_{k|l}$.
The right action of $\mathbf{S}_{\v}(r|s,d)$ on $\mathcal{V}^{k|l}_{r|s}(d)$ is given by
\begin{align*}
\mathop{[A][B]}\limits_{(h\neq m)}
=&\sum_{\substack{i\in\tilde{I}_{k|l}\\a_{ih}\geq1}}(-1)^{(\hat{i}+\widehat{h+1})(\sum_{x>i}\tilde{a}_{xh}+\sum_{x<i}\tilde{a}_{x,h+1})} \v_{h+1}^{\sum_{x<i}(a_{x,h+1}-a_{xh})}[a_{ih}][A-e_{ih}+e_{i,h+1}],\\
\mathop{[A][B]}\limits_{(h=m)}
=&\sum_{\substack{i\in\tilde{I}_{k|l}\\a_{im}\geq1}}(-1)^{\sum_{(x,y)>(i,m)}\tilde{a}_{xy}+(\hat{i}+1)(\sum_{x>i}\tilde{a}_{xm}+\sum_{x<i}\tilde{a}_{x,m+1})}\\& \qquad \times \v_{m+1}^{\sum_{x<i}(a_{x,m+1}+a_{xm})} [a_{im}][A-e_{im}+e_{i,m+1}];
\end{align*}
\begin{align*}
\mathop{[A][C]}_{(h\neq m)}
=&\sum_{\substack{i\in\tilde{I}_{k|l}\\a_{i,h+1}\geq1}}(-1)^{(\hat{i}+\hat{h})(\sum_{x>i}\tilde{a}_{xh}+\sum_{x<i}\tilde{a}_{x,h+1})} \v_{h}^{\sum_{x>i}(a_{xh}-a_{x,h+1})}[a_{i,h+1}][A+e_{ih}-e_{i,h+1}],\\
\mathop{[A][C]}_{(h=m)}
=&\sum_{\substack{i\in\tilde{I}_{k|l}\\a_{i,m+1}\geq1}}(-1)^{\sum_{(x,y)>(i,m+1)}\tilde{a}_{xy}+\hat{i}(\sum_{x>i}\tilde{a}_{xm}+\sum_{x<i}\tilde{a}_{x,m+1})}\\
&\qquad \times \v_{m}^{\sum_{x>i}(a_{xm}+a_{x,m+1})}[a_{i,m+1}][A+e_{im}-e_{i,m+1}].
\end{align*} 
\end{proposition}

\begin{proposition}\label{3.7}
Let $A\in M(k|l,r|s;d)$ and $B,C\in M(k|l,d)$ be such that $\mathrm{co}(B)=\mathrm{co}(C)=\mathrm{ro}(A)$ and $B-e_{h,h+1}$ and $C-e_{h+1,h}$ are diagonal for some $h\in\tilde{I}_{r|s}$. The left action of $\mathbf{S}_{\v}(k|l,d)$ on $\mathcal{V}^{k|l}_{r|s}(d)$ is given by
\begin{align*}
\mathop{[B][A]}\limits_{(h\neq m)}=&\sum_{\substack{j\in\tilde{I}_{r|s}\\a_{h+1,j}\geq1}}\v_{h}^{\sum_{y>j}a_{hy}-{a_{h+1,y}}}[a_{h+1,j}][A+e_{hj}-e_{h+1,j}],\\
\mathop{[B][A]}\limits_{(h=m)}=&\sum_{\substack{j\in\tilde{I}_{r|s}\\a_{m+1,j}\geq1}}(-1)^{\sum_{(x,y)<(m+1,j)}\tilde{a}_{xy}}\v_{m}^{\sum_{y>j}a_{my}+{a_{m+1,y}}}[a_{m+1,j}][A+e_{mj}-e_{m+1,j}];
\end{align*}
\begin{align*}
\mathop{[C][A]}\limits_{(h\neq m)}=&\sum_{\substack{j\in\tilde{I}_{r|s}\\a_{hj}\geq1}}\v_{h+1}^{\sum_{y<j}a_{h+1,y}-{a_{hy}}}[a_{hj}][A-e_{hj}+e_{h+1,j}],\\
\mathop{[C][A]}\limits_{(h=m)}=&\sum_{\substack{j\in\tilde{I}_{r|s}\\a_{mj}\geq1}}(-1)^{\sum_{(x,y)<(m,j)}\tilde{a}_{xy}}\v_{m+1}^{\sum_{y<j}a_{m+1,y}+{a_{my}}}[a_{mj}][A-e_{mj}+e_{m+1,j}].
\end{align*}
\end{proposition}

%------------------------------------
\subsection{From $\mathbf{S}_{\v}(m|n,d)$ to $U_\v(\mathfrak{gl}_{m|n})$}
%------------------------------------

%We will give a map from $U_\v(\mathfrak{gl}_{m|n})$ to $\mathbf{S}_v(m|n)$, then obtain the $(U_\v(\mathfrak{gl}_{k|l}), U_\v(\mathfrak{gl}_{r|s}))$-module $\mathbb{V}^{k|l}_{r|s}$.

%Define $\mathbb{V}^{k|l}_{r|s}(d)=\mathcal{V}^{k|l}_{r|s}(d)\otimes\mathbb{C}(\v)$,\ $\mathbb{V}^{k|l}_{r|s}=\mathop{\oplus}\limits_{d\geq0}\mathbb{V}^{k|l}_{r|s}(d)$.

It was known in \cite{DG14} that there is a surjective homomorphism $U_\v(\mathfrak{gl}_{m|n})\to \mathbf{S}_\v(m|n,d)$ given by
\begin{align*}
&K_h\mapsto\sum_{\lambda\in\Lambda(m+n,d-1)}\v^{\lambda_h}[\mathrm{diag}(\lambda)],\quad K^{-1}_h\mapsto\sum_{\lambda\in\Lambda(m+n,d-1)}\v^{-\lambda_h}[\mathrm{diag}(\lambda)],\\
&E_{h,h+1}\mapsto\sum_{\lambda\in\Lambda(m+n,d-1)}[e_{h,h+1}+\mathrm{diag}(\lambda)],\quad E_{h+1,h}\mapsto\sum_{\lambda\in\Lambda(m+n,d-1)}[e_{h+1,h}+\mathrm{diag}(\lambda)].
\end{align*}
Thus, the $(\mathbf{S}_{\v}(k|l,d), \mathbf{S}_{\v}(r|s,d))$-bimodule $\mathcal{V}^{k|l}_{r|s}(d)$ is also a $(U_\v(\mathfrak{gl}_{k|l}),U_\v(\mathfrak{gl}_{r|s}))$-bimodule through the above surjective homomorphism. 

Denote $$\mathcal{V}^{k|l}_{r|s}:=\bigoplus_{d\geq0}\mathcal{V}^{k|l}_{r|s}(d),$$ 
which has a $\mathbb{C}(\v)$-basis $\{[A] \mid A\in M(k|l,r|s)\}$.
The left $U_\v(\mathfrak{gl}_{k|l})$-action and the right $U_\v(\mathfrak{gl}_{r|s})$-action on $\mathcal{V}^{k|l}_{r|s}$ are given in the following proposition.
\begin{proposition}\label{prop:4.11}
\item[(1)] The left action of $U_\v(\mathfrak{gl}_{k|l})$ on $\mathcal{V}^{k|l}_{r|s}$ is given by, for $h\in\tilde{I}_{r|s}$ and $[A]\in\mathcal{V}^{k|l}_{r|s}(d)$,
\begin{align*}
\mathop{E_{h,h+1}\cdot [A]}\limits_{(h\neq m)}=&\sum_{{j\in \tilde{I}_{s|r}}}\v_{h}^{\sum_{y>j}(a_{hy}-a_{h+1,y})}[a_{h+1,j}][A+e_{hj}-e_{h+1,j}],\\
E_{m,m+1}\cdot [A]=&\sum_{{j\in \tilde{I}_{s|r}}}(-1)^{\sum_{(x,y)>(m+1,j)}\tilde{a}_{xy}}\v_{m}^{\sum_{y<j}(a_{h+1,y}-a_{h,y})}[a_{m+1,j}][A+e_{mj}-e_{m+1,j}],\\
\mathop{E_{h+1,h}\cdot [A]}\limits_{(h\neq m)}=&\sum_{{j\in \tilde{I}_{s|r}}}\v_{h+1}^{\sum_{y<j}(a_{h+1,y}-a_{h,y})}[a_{hj}][A-e_{hj}+e_{h+1,j}],\\
E_{m+1,m}\cdot [A]=&\sum_{{j\in \tilde{I}_{s|r}}}(-1)^{\sum_{(x,y)>(m,j)}\tilde{a}_{xy}}\v_{m+1}^{\sum_{y>j}(a_{m+1,y}+a_{my})}[a_{mj}][A-e_{mj}+e_{m+1,j}].\end{align*}

\item[(2)] The right action of $U_\v(\mathfrak{gl}_{r|s})$ on $\mathcal{V}^{k|l}_{r|s}$ is given by, for $h\in\tilde{I}_{k|l}$ and $[A]\in\mathcal{V}^{k|l}_{r|s}(d)$,
\begin{align*}
\mathop{[A]\cdot E_{h,h+1}}\limits_{(h\neq m)}
=&\sum_{{i\in \tilde{I}_{k|l}}}(-1)^{(\hat{i}+\widehat{h+1})(\sum_{x>i}\tilde{a}_{xh}+\sum_{x<i}\tilde{a}_{x,h+1})}\v_{h+1}^{\sum_{x<i}(a_{x,h+1}-a_{xh})}[a_{ih}][A-e_{ih}+e_{i,h+1}],\\
[A]\cdot E_{m,m+1}
=&\sum_{{i\in \tilde{I}_{k|l}}}(-1)^{\sum_{(x,y)<(i,m)}\tilde{a}_{xy}+(\hat{i}+1)(\sum_{x>i}\tilde{a}_{xm}+\sum_{x<i}\tilde{a}_{x,m+1})}\\
&\qquad \times \v_{m+1}^{\sum_{x<i}(a_{x,m+1}+a_{xm})}[a_{im}][A-e_{im}+e_{i,m+1}],\\
\mathop{[A]\cdot E_{h+1,h}}\limits_{(h\neq m)}
=&\sum_{{i\in \tilde{I}_{k|l}}}(-1)^{(\hat{i}+\hat{h})(\sum_{x>i}\tilde{a}_{xh}+\sum_{x<i}\tilde{a}_{x,h+1})}\v_{h}^{\sum_{x>i}(a_{xh}-a_{x,h+1})}[a_{i,h+1}][A+e_{ih}-e_{i,h+1}],\\
[A]\cdot E_{m+1,m}
=&\sum_{{i\in \tilde{I}_{k|l}}}(-1)^{\sum_{(x,y)<(i,m+1)}\tilde{a}_{xy}+\hat{i}(\sum_{x>i}\tilde{a}_{xm}+\sum_{x<i}\tilde{a}_{x,m+1})}\\
&\qquad \times \v_{m}^{\sum_{x>i}(a_{xm}+a_{x,m+1})}[a_{i,m+1}][A+e_{im}-e_{i,m+1}].
\end{align*}
\end{proposition}

%----------------------------------
\subsection{Equivalence of the three constructions}
%-------------------------------
As we mentioned, the three bimodules $\mathcal{M}^{k|l}_{r|s}$, $\mathcal{S}^{k|l}_{r|s}$ and $\mathcal{V}^{k|l}_{r|s}$ all have a basis indexed by $M(k|l,r|s)$, which implies that they are isomorphic as $\mathbb{C}(\v)$-spaces. By comparing Propositions~\ref{coor}, \ref{prop:3.3} and \ref{prop:4.11}, we arrive at the following conclusion. 
\begin{theorem}
As $(U_\v(\mathfrak{gl}_{k|l}),U_\v(\mathfrak{gl}_{r|s}))$-bimodules, we have
$\mathcal{M}^{k|l}_{r|s}\simeq\mathcal{S}^{k|l}_{r|s}\simeq\mathcal{V}^{k|l}_{r|s}$.
\end{theorem}

Therefore, we can rewrite Theorem~\ref{Howe} as follows. 
\begin{corollary}
The three $(U_\v(\mathfrak{gl}_{k|l}),U_\v(\mathfrak{gl}_{r|s}))$-bimodules $\mathcal{M}^{k|l}_{r|s}$, $\mathcal{S}^{k|l}_{r|s}$ and $\mathcal{V}^{k|l}_{r|s}$ all afford a double centralizer property between $U_\v(\mathfrak{gl}_{k|l})$ and $U_\v(\mathfrak{gl}_{r|s})$, and hence have a multiplicity-free decomposition :
\begin{equation*}
\mathcal{M}^{k|l}_{r|s}\simeq\mathcal{S}^{k|l}_{r|s}\simeq\mathcal{V}^{k|l}_{r|s}\simeq\bigoplus_{\lambda\in\tilde{\Lambda}(k|l)\cup\tilde{\Lambda}(r|s)}L^{k|l}_{\lambda}\otimes \widetilde{L}^{r|s}_{\lambda}.
\end{equation*}
\end{corollary}

\vspace{1cm}

\noindent{\bf Disclosure Statement.} There are no relevant financial or non-financial competing interests to report.

%====================================================================== 

\end{document}